\documentclass{article}
\usepackage[utf8]{inputenc}
\usepackage{amstext,amssymb,amsmath,amsbsy}
\usepackage{hyperref}
\usepackage{amscd}
\usepackage{amsfonts}
\usepackage{indentfirst}
\usepackage{verbatim}
\usepackage{amsmath}
\usepackage{amsthm}
\usepackage{enumerate}
\usepackage{graphicx}
\usepackage{color, soul}
\usepackage[OT1]{fontenc}
\usepackage{subfig}
\usepackage{algorithm}
\usepackage{algorithmic}

\usepackage{mathtools}
\DeclarePairedDelimiter\ceil{\lceil}{\rceil}

\newcommand{\R}{\mathbb{R}}

\newcommand{\ds}{\displaystyle}
\newcommand{\Id}{\textrm{Id}}
\newcommand{\x}{{\bf x}}
\newcommand{\y}{{\bf y}}
\newcommand{\p}{{\bf p}}

\newcommand{\Div}{{\rm div}}

\newtheorem{Theorem}{Theorem}[section]

\newtheorem{Proposition}{Proposition}[section]
\newtheorem{Corollary}{Corollary}[section]
\newtheorem{Remark}{Remark}[section]

\newtheorem*{Assumption*}{Assumption}
\newtheorem{Definition}{Definition}[section]
\newtheorem{problem}{Problem}[section]
\newtheorem*{problem*}{Problem}
\numberwithin{equation}{section}

\title{ The Carleman convexification method for Hamilton-Jacobi equations  on the whole space}
\author{Huynh P.N. Le\thanks{Faculty of Mathematics and Computer Science, University of Science, Vietnam National University, Ho Chi Minh City, Vietnam}
\and
Thuy T. Le\thanks{
Department of Mathematics and Statistics, University of North Carolina at
Charlotte, Charlotte, NC 28223, USA. 
}  \and Loc H. Nguyen\footnotemark[2] \thanks{Corresponding author \texttt{loc.nguyen@uncc.edu.}}}

\date{}

\begin{document}

\maketitle
\begin{abstract}
We propose a new globally convergent numerical method  to solve  Hamilton-Jacobi equations in $\R^d$, $d \geq 1$.
This method is named as the Carleman  convexification method.
By Carleman  convexification, we mean that we use a Carleman weight function to convexify the conventional least squares mismatch functional. 
We will prove a new version of the convexification theorem guaranteeing that the mismatch functional involving the Carleman weight function is strictly convex 
and, therefore, has a unique minimizer.
Moreover, a consequence of our convexification theorem guarantees that the minimizer of the Carleman weighted mismatch functional is an approximation of the viscosity solution we want to compute. 
Some numerical results in 1D and 2D will be presented.


\end{abstract}

\noindent{\it Key words: numerical methods; Carleman estimate; 
 Hamilton-Jacobi equations; viscosity solutions; vanishing viscosity process.} 

\noindent{\it AMS subject classification:
35D40, 
35F21
} 	
\section{Introduction}\label{sec1}

Let $d \geq 1$ be the spatial dimension.
Let $H: \mathbb{R}^d \times  \mathbb{R}^d \to \mathbb{R}$ be a function  satisfying the following growth condition 
\begin{equation}
	|H(\x, \p)| \leq C|\p|^{k}
	\quad
	\mbox{for all }
	\p \in \R^d
	\label{growth}
\end{equation}
for some number $k > 0.$
In this paper, we solve the following problem.
\begin{problem}
Fix $\lambda > 0$. 
Assume that equation 
\begin{equation}
    \lambda u +  H({\bf x}, \nabla u) = 0 \quad 
    \mbox{for all } {\bf x} \in \mathbb{R}^d
    \label{HJ}
\end{equation}
has a  unique viscosity solution $u$. 
Compute $u$.
\label{p1}
\end{problem}

In general, the condition in the problem statement above requiring that \eqref{HJ} has a unique viscosity solution might not always hold true. 
We provide an example of a set of conditions on $H$ such that \eqref{HJ} has a unique solution. 
If $H$ is such that $|H(\x, \p) - H(\y, \p)| \leq C(1 + |\p|)|\x - \y|$ and $|H(\x, \p) - H(\x, {\bf q})| \leq C|\p - {\bf q}| $ for some positive constant $C$ 
for all $\x,$ $\y$, $\p$, ${\bf q}$ in $\R^d,$
then the comparison principle in \cite[Theorem 1.18]{Hung:book2021} is valid.
  The uniqueness follows directly. 
The existence of solution to \eqref{HJ} is studied in \cite[Chapter 2]{Hung:book2021}. 
We refer the reader to
\cite{BCD, Barles, CrandallEvansLions84, CrandallLions83, Lions, Hung:book2021} for more important and interesting theory about Hamilton-Jabobi equation.
For  convenience, we recall, from the pioneer works \cite{CrandallEvansLions84, CrandallLions83}  as well as the recent published book \cite{Hung:book2021}, the concept of viscosity solutions to Hamilton-Jacobi equations. Viscosity sub(super)-solutions  to \eqref{HJ} are defined as follows.

\begin{Definition}[Viscosity solutions]
Let $F: \R^d \times \R \times \R^d \to \R$ be an Hamiltonian. 
Let $u \in C(\R^d)$.
\begin{itemize}
\item We say that $u$ is a viscosity subsolution to $F(\x, u, \nabla u) = 0$ if for any test function $\varphi \in C^1(\R^d)$ such that $u-\varphi$ has a strict maximum at $\x_0 \in \R^d$, then
\[
F(\x_0, u(\x_0), \nabla \varphi(\x_0)) \leq 0 \quad \text{ if } \x_0 \in \R^d.
\]

\item We say that $u$ is a viscosity supersolution to $F(\x, u, \nabla u) = 0$ if for any test function $\varphi \in C^1(\R^d)$ such that $u-\varphi$ has a strict minimum at $\x_0 \in \R^d$, then
\[
F(\x_0, u(\x_0),  \nabla \varphi(\x_0)) \geq 0 \quad \text{ if } \x_0 \in \R^d.
\]

\item We say that $u$ is a viscosity solution to $F(\x, u, \nabla u) = 0$ if it is both viscosity subsolution and  viscosity supersolution to this equation.
\end{itemize}
\label{def_viscos}
\end{Definition}

A number of  efficient and fast numerical approaches and techniques (many of which are of high orders) have been developed for Hamilton-Jacobi equations of the form $F(\x, u, \nabla u) = 0$ where $F$ is called the Hamiltonian.
For finite difference monotone and consistent schemes of first-order equations and applications, see \cite{BS-num, CL-rate,  OsFe, Sethian,  Sou1} for details and recent developments.
If $F=F(\x,s,\p)$ is convex in $\p$ and satisfies some appropriate conditions, it is possible to construct some semi-Lagrangian approximations by the discretization of the Dynamical Programming Principle associated to the problem, see \cite{FaFe1, FaFe2} and the references therein.
See \cite{Abgrall, Abgrall2, BrysonLevy, CagnettiGomesTran, CamilliCDGomes, CockburnMerevQian, GallistlSprekelerSuli, KaoOsherQian, LiQian,  ObermanSalvador, OsherSethian, OsherShu,   QianZhangZhao, SethianVladimirsky, TsaiChengOsherZhao, Tsitsiklis} for an incomplete list of results in this directions.
Another approach to solve  Hamilton-Jacobi equations is based on optimization \cite{DanielDurou, HornBrooks, LeclercBobick, Szeliski}. 
However, due to the nonlinearity of the Hamiltonian, the least squares cost functional is nonconvex and might have multiple local minima and ravines.
Hence, the methods based on optimization can provide reliable numerical solutions if good initial guesses of the true solutions are given. 
The key point of the convexification method in this paper is to include in such least  squares mismatch functionals some Carleman weight functions to make these functionals convex.
Therefore, the requirement about the good initial guess is completely relaxed.
On the other hand, we especially draw the reader attention to \cite{Abgrall,  OsFe, OsherShu} for the Lax--Friedrichs schemes and \cite{KaoOsherQian,LiQian} for the Lax--Friedrichs sweeping algorithm to solve Hamilton-Jacobi equations. 
Although strong, these methods might not be applicable  to solve \eqref{HJ}. The main reason is that in computation, rather than  finding a solution to \eqref{HJ} on the whole space $\R^d$,
one can compute the restriction of the solution to \eqref{HJ} on a bounded domain. In this case, the boundary conditions on the boundary of this bounded domain are unclear while the sweeping methods are initiated by the boundary conditions of solutions.

Recently, we, in \cite{KhoaKlibanovLoc:SIAMImaging2020, KlibanovNguyenTran:JCP2022}, developed the convexification method to solve (1) the inverse scattering problems and (2) a general class of Hamilton-Jacobi equations in a bounded domain. The efficiency of the convexification method was rigorously proved. However, the version of the convexification method above requires both Neumann and Dirichlet boundary conditions of the solution, which are not always available. Therefore, in order to apply the convexification method in \cite{KhoaKlibanovLoc:SIAMImaging2020, KlibanovNguyenTran:JCP2022} without requesting the knowledge of the boundary conditions, we have to develop a new version. 
The key of the success involves a new piece-wise Carleman estimate and a new mismatch functional with a suitable Carleman weight function.
Our method to solve \eqref{HJ} consists of two stages.
In stage 1, we apply a truncation technique to reduce the problem of solving \eqref{HJ} on the whole $\R^d$ to the problem of computing viscosity of another Hamilton-Jacobi equation on a bounded domain. The boundary conditions of the new Hamilton-Jacobi equation are unknown but we can estimate them in term of the cut-off function.
Then, in stage 2, we minimize a mismatch functional with a special Carleman weight function involved, called the Carleman weighted mismatch functional. 
The presence of the Carleman weight function is extremely important in the sense that it guarantees  the strict convexity of the Carleman weighted mismatch functional.  As a result, our  Carleman weighted mismatch functional  has a unique minimizer in any bounded set of the functional space under consideration.
We will apply a Carleman estimate to prove this theoretical result, called a convexification theorem. 
Besides guaranteeing the strict convexity of the cost functional, the convexification theorem can be used to prove that the minimizer is an approximation of the desired viscosity solution.

It is worth to mention that several versions of the convexification method have
been developed since it was first introduced in \cite%
{KlibanovIoussoupova:SMA1995} for a coefficient inverse problem for a hyperbolic equation.
 We cite here \cite{KlibanovNik:ra2017, VoKlibanovNguyen:IP2020, 
KhoaKlibanovLoc:SIAMImaging2020,
Klibanov:sjma1997, 
Klibanov:nw1997, 
Klibanov:ip2015, 
Klibanov:IPI2019,
KlibanovLeNguyenIPI2021,    
Klibanov:ip2020,
LeKlibanov:ip2022,
SmirnovKlibanovNguyen:IPI2020} and references therein
for some important works in this area and their real-world applications in bio-medical imaging, non-destructed testing, 
travel time tomography, identifying anti-personnel explosive devices buried under the ground, {\it etc.} 
The crucial mathematical ingredient  that
guarantees the strict convexity of this functional is the presence of some
Carleman estimates. 
The original idea of applying Carleman estimates to prove the uniqueness for a large class of important nonlinear mathematical problems was first published in \cite{BukhgeimKlibanov:smd1981}. 
It was discovered later in \cite{KlibanovIoussoupova:SMA1995, KlibanovLiBook}, that the idea of \cite%
{BukhgeimKlibanov:smd1981} can be successfully modified to develop globally
convergent numerical methods for coefficient inverse problems using the
convexification method.

The structure of the paper. 
In Section \ref{sec_changevariable}, we reduce the problem of computing solution to \eqref{HJ} to the problem of computing viscosity solution to another Hamilton-Jacobi equation on a bounded domain of $\R^d$.
In Section \ref{sec_Car}, we prove a Carleman estimate, which plays a key role in the proof of the convexification theorem.
In Section \ref{sec_convexi}, we prove the convexification theorem.
In Section \ref{sec5}, we show some numerical examples.
Section \ref{sec6} is for the concluding remarks.

\section{A change of variable}\label{sec_changevariable}

Rather than computing the solution $u$ to \eqref{HJ} on the whole space $\R^d$, we compute the restriction of $u$ on an arbitrary bounded domain $G$ of $\R^d$. 
Without lost of generality, we assume that $G$ is compactly contained inside the cube
 $\Omega = (-R, R)^d$ for some $R > 0.$
Let $\delta \in (0, 1)$ be a small number. 
Let $\chi_\delta$ be a cut off function in the class $C^\infty(\mathbb{R}^d)$ satisfying 
\begin{equation}
    \chi_\delta(\x) = 
    \left\{
        \begin{array}{ll}
            > c > 0 & \x \in G  \\
            \in (\delta, c) & \x \in \Omega \setminus G,\\
            = \delta &\x \in \R^d \setminus \Omega
        \end{array}
    \right.
    \label{chi}
\end{equation}
for some constant $c > 0$.
Define 
\begin{equation}
	v(\x) = \chi_\delta(\x) u(\x)
	\quad
	\mbox{or equivalently}
	\quad
	u(\x) = \frac{v(\x)}{\chi_{\delta}(\x)}
	\label{2,1}
\end{equation}
for all $\x \in \R^d.$
Since
\begin{align*}
	\lambda u(\x) +  H({\bf x}, \nabla u(\x)) &= \lambda \frac{v(\x)}{\chi_\delta} + H\left(\x,\nabla \frac{v(\x)}{\chi_\delta(\x)}\right)\\
	&= \lambda \frac{v(\x)}{\chi_\delta(\x)} + H\left(\x,\frac{\chi_\delta(\x)\nabla v(\x) - v(\x)\nabla \chi_\delta(\x)}{\chi_\delta^2(\x)}\right),
\end{align*}
it follows from \eqref{HJ} that
\begin{equation}
    \lambda \frac{v(\x)}{\chi_\delta(\x)} + H\left(\x,\frac{\chi_\delta(\x)\nabla v(\x) - v(\x)\nabla \chi_\delta(\x)}{\chi_\delta^2(\x)}\right) = 0
    \label{2,,3}
\end{equation}
for all $\x \in \R^d.$
Multiplying $\chi_{\delta}^{2k}(\x)$ to both sides of \eqref{2,,3}, we  derive  an equation for $v$, read as
\begin{equation}
	F(\x, v(\x), \nabla v(\x)) := \chi_{\delta}^{2k}(\x)\left[\lambda  \frac{v(\x)}{\chi_\delta(\x)} +  H\left(\x,\frac{\chi_\delta(\x)\nabla v(\x) - v(\x)\nabla \chi_\delta(\x)}{\chi_\delta^2(\x)}\right) \right]= 0
	\label{HJv}
\end{equation}
for all $\x \in \R^d.$
\begin{Remark}
	The presence of $\chi_{\delta}^{2k}(\x)$ in the the right hand side of \eqref{HJv} helps us remove the blow-up behavior of the term $\frac{\chi_\delta(\x)\nabla v(\x) - v(\x)\nabla \chi_\delta(\x)}{\chi_\delta^2(\x)}$ as $\delta \to 0^+$. In fact, by \eqref{growth},
	\begin{align*}
		\left| \chi_{\delta}^{2k}(\x) H\left(\x,\frac{\chi_\delta(\x)\nabla v(\x) - v(\x)\nabla \chi_\delta(\x)}{\chi_\delta^2(\x)}\right)\right|
		&\leq C  \chi_{\delta}^{2k}(\x) \left|\frac{\chi_\delta(\x)\nabla v(\x) - v(\x)\nabla \chi_\delta(\x)}{\chi_\delta^2(\x)}\right|^{k}
		\\
		&= C|\chi_\delta(\x)\nabla v(\x) - v(\x)\nabla \chi_\delta(\x)|^k,
	\end{align*}
	which is uniformly bounded provided that $v \in C^1(\R^d).$ 
	This step  is crucial in numerical computation.
\end{Remark}

We have the proposition.

\begin{Proposition}
    The function $u$ is a viscosity solution to \eqref{HJ} if and only if the function $v$ is a viscosity solution to \eqref{HJv}.
    \label{thm1}
\end{Proposition}
\begin{proof}
    
 We prove that if $u$ is a viscosity subsolution to \eqref{HJ} then $v$ is a viscosity solution to \eqref{HJv}.
In fact, fix an arbitrary point $\x_0 \in \R^d$ and let $\varphi$ be a test function such that $v - \varphi$ has a strict maximum at $\x_0$. 
We have
\begin{equation}
	v(\x) - \varphi(\x) - [v(\x_0) - \varphi(\x_0)] < 0,
	\mbox{for all } \x \in \R^d \setminus \{\x_0\}.
	\label{2,3}
\end{equation}
Define 
\begin{equation}
	\phi(\x) =\frac{\varphi(\x) + (v(\x_0) - \varphi(\x_0))}{ \chi_\delta(\x)}
	\quad
	\mbox{for all } \x \in \R^d.
	\label{2.3}
\end{equation}
Then, $\phi \in C^1(\R^d)$.  
A simple algebra yields that for all $\x \in \R^d,$
\begin{align*}
	 \chi_\delta^{2k}(\x)& \Big[
		\lambda u(\x) + H(\x, \nabla \phi(\x))
	\Big]
	= \chi_\delta^{2k}(\x) \Big[
		\lambda  \frac{v(\x)}{\chi_\delta(\x)} + H\Big(\x, \nabla \frac{\varphi(\x) + (v(\x_0) - \varphi(\x_0))}{ \chi_\delta(\x)}\Big)
	\Big]
	\\
	&
	= \chi_\delta^{2k}(\x) \Big[
		\lambda \frac{v(\x)}{\chi_\delta(\x)} + H\Big(\x, \frac{ \chi_{\delta}(\x) \nabla \big(\varphi(\x) + (v(\x_0) - \varphi(\x_0))\big) - \big(\varphi(\x) + (v(\x_0) - \varphi(\x_0))\big) \nabla \chi_{\delta}(\x)}{\chi_\delta^2(\x)}\Big)
	\\
	&=\chi_\delta^{2k}(\x) \Big[
		\lambda \frac{v(\x)}{\chi_\delta(\x)} + H\Big(\x,  \frac{ \chi_{\delta}(\x) \nabla (\varphi(\x) ) - \big(\varphi(\x) + (v(\x_0) - \varphi(\x_0))\big) \nabla \chi_{\delta}(\x_0)}{\chi_\delta^2(\x)}\Big)\Big].
\end{align*}
In particular, when $\x = \x_0$, we have
\begin{align*}
	\chi_\delta^{2k}(\x_0) \Big[
		\lambda u(\x_0) & + H(\x_0, \nabla \phi(\x_0))
	\Big]
	\\
	&=
	\chi_\delta^{2k}(\x_0) \Big[
		\lambda \frac{v(\x_0)}{\chi_\delta(\x_0)} + H\Big(\x_0,  \frac{ \chi_{\delta}(\x_0) \nabla (\varphi(\x_0) ) - \big(\varphi(\x_0) + (v(\x_0) - \varphi(\x_0))\big) \nabla \chi_{\delta}(\x_0)}{\chi_\delta^2(\x)}\Big)\Big]
		\\
	&=\chi_\delta^{2k}(\x_0) \Big[
		\lambda \frac{v(\x_0)}{\chi_\delta(\x_0)} + H\Big(\x_0,  \frac{ \chi_{\delta}(\x_0) \nabla (\varphi(\x_0) ) - v(\x_0) \nabla \chi_{\delta}(\x_0)}{\chi_\delta^2(\x)}\Big)\Big]
\end{align*}
Hence, by \eqref{HJv},
\begin{equation}
	F(\x_0, v(\x_0), \nabla \varphi(\x_0)) = \chi_\delta^{2k}(\x_0) \Big[
		\lambda u(\x_0) + H(\x_0, \nabla \phi(\x_0))
	\Big] .
		 \label{2.5}
\end{equation}
Due to  \eqref{2,1} and \eqref{2.3}, for all $\x \in \R^d \setminus \{\x_0\}$,
\begin{align*}
	u(\x) - \phi(\x) 
	&=\frac{v(\x) - \big(\varphi(\x) + (v(\x_0) - \varphi(\x_0))\big)}{\chi_{\delta}(\x)}
	\\
	&= \frac{v(\x) - \varphi(\x) - (v(\x_0) - \varphi(\x_0))}{\chi_{\delta}(\x)} < 0.
\end{align*}
It is obvious that $u(\x_0) - \phi(\x_0) = 0.$ 
Hence, $u - \phi$ attains a strict maximum at $\x_0$. 
Since $u$ is a viscosity subsolution to \eqref{HJ},
\begin{equation}
	\chi_\delta^{2k}(\x_0) \Big[
		\lambda u(\x_0) + H(\x_0, \nabla \phi(\x_0))
	\Big]   
	\leq 0.
	\label{2.6}
	\end{equation}
	Combining \eqref{2.5} and \eqref{2.6}, we obtain
	\[
		F(\x_0, v(\x_0), \nabla \varphi(\x_0)) \leq 0.
	\]
Hence $v$ is a viscosity subsolution to \eqref{HJv}. 
We can repeat the proof above to show that if $u$ is a viscosity supersolution to \eqref{HJ} then $v$ is a viscosity supersolution to \eqref{HJv}. The reverse direction of Theorem \ref{thm1} can be proved in the same manner.
\end{proof}

\begin{Remark}
A direct consequence of Proposition \ref{thm1} is that we can compute the viscosity solution $u^*$ to \eqref{HJ} by finding the viscosity solution $v^*$ to  \eqref{HJv} and then setting $u^* = v^*/\chi_\delta$.
	Although the formula $u^* = v^*/\chi_\delta$ holds true for all $\x \in \R^d$, this formula is reliable only in the domain $G$ where $\chi_{\delta} > c > 0$. Outside $G$, the function $\chi_\delta$ is close to zero. In this case, the ``artificial" error in computation, due to discretization with positive step size, the presence of the viscosity term and regularization term, is magnified. 
\end{Remark}

It is well-known from the vanishing viscosity process  that $v^*$ can be approximated by the solution to
\begin{equation}
    -\epsilon_0 \Delta v_{\epsilon_0} + F(\x, v_{\epsilon_0}, \nabla v_{\epsilon_0}) = 0
    \quad
    \mbox{for all } \x \in \R^d.
    \label{2.9}
\end{equation}
In computation, it is inconvenient to compute the function $v_{\epsilon_0}$ on the whole space $\R^d$.
We only find $v_{\epsilon_0}$ in the bounded domain $\Omega$ on which $\chi_{\delta} > \delta.$
In order to solve PDEs of the form \eqref{2.9} on a bounded domain, we have to approximate the boundary conditions. 
By \eqref{chi} and \eqref{2,1}, for all $\x \in \partial \Omega$,
\begin{equation}
	|v(\x)| = \delta |u(\x)| < C\delta
	\label{2.1010}
\end{equation} 
and
\begin{equation}
	|\partial_{\nu} v(\x)| = |\chi_\delta(\x)\partial_{\nu} u(\x) + u(\x) \partial_{\nu}\chi_{\delta}(\x)|
	=|\chi_\delta(\x)\partial_{\nu} u(\x)| < C\delta.
	\label{2.1111}
\end{equation}
Here, we have used the fact that $\partial_{\nu}\chi_{\delta}(\x)|_{\partial \Omega} = 0$. 
	Since the functions $u|_{\partial \Omega}$ and $\partial_{\nu}u|_{\partial \Omega}$ are unknown, neither are $v|_{\partial \Omega}$ and $\partial_{\nu}v|_{\partial \Omega}$.
We are unable to compute the exact Cauchy information of $v$ on $\partial \Omega$. However, 
since $\delta$ is a small number, due to \eqref{2.1010} and \eqref{2.1111}, both $|v|$ and $|\nabla v|$ on $\partial \Omega$ are small. 
Therefore, we can impose the conditions that both $v_{\epsilon_0}|_{\partial \Omega}$ and $\nabla v_{\epsilon_0}|_{\partial \Omega}$ satisfy
\begin{equation}
    |v_{\epsilon_0}| < C\delta 
    \quad 
    \mbox{and } 
    |\nabla v_{\epsilon_0}| < C\delta
    \label{2.12}
\end{equation}
for all $\x \in \partial \Omega.$

	Since the exact boundary condition for $v_{\epsilon_0}$ cannot be retrieved, numerical methods to compute it is not yet developed. 
    Conventional methods compute a function $v_{\epsilon_0}$ that satisfies \eqref{2.9} and \eqref{2.12} are based on least squares optimization. 
    That means we minimize a cost functional and then set the minimizer, named as $v_{\rm min}^{\epsilon_0}$, as the computed solution. 
    A typical example of such a functional is
    \begin{equation}
        v \mapsto J(v) = \int_{\Omega}|-\epsilon_0 \Delta v + F(\x, v, \nabla v)|^2 d\x
        + \int_{\partial \Omega} |v|^2d\x
        + \int_{\partial \Omega} |\nabla v|^2d\sigma(\x)
        + \mbox{a regularization term}.
        \label{2.131313}
    \end{equation}
This approach is effective in many cases. It is widely used in the scientific community.
However, it has several drawbacks. 
The most important drawback is that finding the global minimizer $v_{\rm min}^{\epsilon_0}$ is extremely challenging unless a  good initial guess is given.
This is because the functional $J$  might not be convex and might have multiple local minima. 
The second drawback is that,  in general,  the distance between the true solution $v_{\epsilon_0}$ to \eqref{HJv} and the computed solution $v_{\rm min}^{\epsilon_0}$ is not known.
In this paper, we generalize the convexification method in \cite{KhoaKlibanovLoc:SIAMImaging2020, KlibanovNguyenTran:JCP2022} to compute the ``best fit" solution to \eqref{2.9} and \eqref{2.12}.
By convexification, we mean that we let a Carleman weight function involve in the functional $J$, defined in \eqref{2.131313}. 
The presence of the Carleman weight function remove both significant drawbacks of the least squares optimization approach above.
As mentioned in Section \ref{sec1}, the idea of using Carleman weight function to convexify the functional $J$ was originally introduce in \cite{KlibanovIoussoupova:SMA1995} and then was investigated intensively by our research group, see e.g., \cite{KlibanovNik:ra2017, KhoaKlibanovLoc:SIAMImaging2020, KlibanovNguyenTran:JCP2022,  LeNguyen:JSC2022}.

\section{A piece-wise Carleman estimate}\label{sec_Car}

The key tool for us to rigorously prove the convexifying phenomenon is the Carleman estimate established in this section. 
Let $\Omega$ be a bounded domain in $\R^d$ with smooth boundary. 
Let $A: \overline \Omega \to \R^{d \times d}$ be a $d \times d$ matrix valued function in the class $C^2$. 
Assume that 
\begin{enumerate} 
\item $A$ is symmetric; i.e., $A^{\rm T} = A$ or equivalently $a_{ij} = a_{ji}$ for all $1 \leq i, j \leq d$, where $a_{ij}$ is the entry on row $i$ and column $j$ of $A$;
\item $A$ is positive definite; i.e., there exists a positive number $\Lambda$ such that
\begin{equation}
   \Lambda^{-1} |\xi|^2 \leq A(\x) \xi \cdot \xi \leq \Lambda |\xi|^2 \quad \mbox{for all } \x \in \overline \Omega, \xi \in \R^d.
    \label{ellipic_matrix}
\end{equation}
\end{enumerate}
Let $\x_0$ be a point in $\R^d \setminus \Omega.$ For each $\x \in \R^d$, define 
\begin{equation}
    r(\x) = |\x - \x_0|
    \quad \mbox{for all } \x \in \overline \Omega.
\end{equation}
We have the theorem.
\begin{Theorem}
    Let $\lambda > 0$ and $u \in C^2(\overline \Omega).$
    Then, there exists a positive constant $\beta_0$ depending only on $\|A\|_{C_1(\overline \Omega)}$ and $\Lambda$ 
    such that if $\beta \geq \beta_0$ and $\lambda \geq \lambda_0 = 2R^\beta$, where $R = \max_{\x \in \overline \Omega}\{|\x - \x_0|\}$, then
    \begin{equation}
        r^{\beta + 2} e^{2\lambda r^{-\beta}} |\Div(A\nabla u)|^2  
    \geq C\Big[
        \Div (U) + \lambda^3\beta^4e^{2\lambda r^{-\beta}} r^{-2\beta - 2} |u|^2
        + \lambda \beta e^{2\lambda r^{-\beta}} |\nabla u|^2 
    \Big].
    \label{CarEst}
    \end{equation}
    Here, $U$ is  a vector valued function satisfying
    \begin{equation}
        |U| \leq Ce^{2\lambda r^{-\beta}}(\lambda^3 \beta^3 r^{-2\beta - 2}|u|^2 + \lambda \beta |\nabla u|^2)
        \label{divU}
    \end{equation} and $C$ is a constant depending only on $\x_0,$ $\Omega$, $\|A\|_{C^1(\overline \Omega)}$, $\Lambda$ and $d$. 
    \label{thmCarpointwise}
\end{Theorem}

\begin{proof}[Proof of Theorem \ref{thmCarpointwise} ]
In the proof, we denote by $C_i$, $i \in \{1, 2,  \dots\}$, positive constants depending only on $\|A\|_{C^1(\overline \Omega)},$ $\x_0, $ $\Omega$, $\Lambda$ and $d.$
We split the proof into several steps.

\noindent{\bf Step 1}.
For $\x \in \Omega$, recall $r = |\x - \x_0|.$
Set
\begin{equation}
    v = e^{\lambda r^{-\beta}}u
    \quad
    \mbox{equivalently}
    \quad
    u = e^{-\lambda r^{-\beta}}v.
    \label{2.10}
\end{equation}
By the product rule in differentiation and the symmetry of $A$, we have
\begin{align*}
    \Div (A\nabla u) &= \Div (A \nabla (e^{-\lambda r^{-\beta}} v))
    %
    \\
    &= 2A\nabla v \cdot \nabla (e^{-\lambda r^{-\beta}})
    + e^{-\lambda r^{-\beta}}\Div (A \nabla v)
    + v \Div (A\nabla e^{-\lambda r^{-\beta}})\\
    &= 2\lambda \beta r^{-\beta - 2} e^{-\lambda r^{-\beta}} A\nabla v \cdot (\x - \x_0)
    + e^{-\lambda r^{-\beta}}\Div (A \nabla v)
    + v \Div (A\nabla e^{-\lambda r^{-\beta}}).
\end{align*}
Using the inequality $(a + b + c)^2 \geq 2a(b + c),$ we have for all $\x \in \Omega$
\begin{equation}
    \frac{r^{\beta + 2} e^{2\lambda r^{-\beta}}|\Div (A\nabla u)|^2}{2\lambda \beta} 
    \geq 
    2 (\x - \x_0) \cdot A\nabla v 
    \Div (A \nabla v)
     + 2 e^{\lambda r^{-\beta}} [(\x - \x_0) \cdot A\nabla v ]
     v \Div (A\nabla e^{\lambda r^{-\beta}}).
     \label{2.11}
\end{equation}
Denote by
\begin{align}
    I_1 &= 2 (\x - \x_0) \cdot A\nabla v \Div (A \nabla v) \label{I1}
    \\
    I_2 &= 2e^{\lambda r^{-\beta}} (\x - \x_0) \cdot A\nabla (|v|^2) 
      \Div (A\nabla e^{\lambda r^{-\beta}}). \label{I2}
\end{align}
Due to  \eqref{I1} and \eqref{I2}, we rewrite \eqref{2.11} as
\begin{equation}
    \frac{r^{\beta + 2} e^{2\lambda r^{-\beta}}|\Div (A\nabla u)|^2}{2\lambda \beta} \geq I_1 + I_2
    \label{2.141414}
\end{equation}
for all $\x \in \Omega.$
We next estimate $I_1$ and $I_2.$

\noindent {\bf Step 2.}  In this step, we estimate $I_1$.
By the product rule in differentiation $f \Div F = \Div (f F) - \nabla f \cdot F$ for all  scalar valued function $f$ and vector valued function $F$, we have
\begin{align*}
    I_1 &= 2(\x - \x_0) \cdot A\nabla v 
    \Div (A \nabla v)
    \\
    &=2 \Div \big( [(\x - \x_0) \cdot A\nabla v]   A \nabla v\big) - 2 \nabla ((\x - \x_0) \cdot A\nabla v  ) \cdot (A \nabla v)
\end{align*}
for all $\x \in \Omega.$
Thus,
\begin{equation}
    I_1 = \Div V_1 - 2 \nabla ((\x - \x_0) \cdot A\nabla v)  \cdot (A \nabla v)
    \label{c_2_2.9}
\end{equation}
where $V_1$ is the vector defined by
\begin{equation} 
    V_1 = 2[(\x - \x_0) 
    \cdot A\nabla v ]  A \nabla v.
    \label{V1}
\end{equation}
Using the symmetry of $A = (a_{ij})_{j = 1}^d$, we have
\begin{align}
    &\frac{\partial}{\partial x_i} ((\x - \x_0)\cdot A\nabla v  ) 
    = \frac{\partial}{\partial x_i} (A(\x - \x_0)\cdot \nabla v  ) 
     = \frac{\partial}{\partial x_i} \Big(
        \sum_{k, j = 1}^d (x_j - (x_0)_j) a_{kj}  \frac{\partial v}{\partial x_k}  
     \Big) \notag
    \\
    &\hspace{0.05cm} = \sum_{k, j = 1}^d \Big[
        (x_j - (x_0)_j) a_{kj}  \frac{\partial^2 v}{\partial x_k \partial x_i} 
        + a_{kj} \delta_{ij}  \frac{\partial v}{\partial x_k}
        + (x_j - (x_0)_j) \frac{\partial a_{kj}}{\partial x_i} \frac{\partial v}{\partial x_k}
    \Big].
    \label{c_2_2.11}
\end{align}
Here,
$
    \delta_{ij} = \left\{
        \begin{array}{ll}
             1&  i = j\\
             0&  i \not = j
        \end{array}
    \right.
$
is called the Kronecker delta and
$x_j$ and $(x_0)_j$ are the $j^{\rm}$ entries of $\x$ and $\x_0$ respectively.
By writing
\begin{multline*}
    2\sum_{i, j,k,l = 1}^d (x_j - (x_0)_j) a_{kj}  \frac{\partial^2 v}{\partial x_k \partial x_i}a_{il} \frac{\partial v}{\partial l} 
    = \sum_{i, j,k,l = 1}^d (x_j - (x_0)_j) a_{kj}  \frac{\partial^2 v}{\partial x_k \partial x_i}a_{il} \frac{\partial v}{\partial l}\\
    + \sum_{i, j,k,l = 1}^d (x_j - (x_0)_j) a_{kj}  \frac{\partial^2 v}{\partial x_k \partial x_i}a_{il} \frac{\partial v}{\partial l},
\end{multline*}
and by interchanging the roles of the indices $i$ and $l$ in the second sum, we obtain
\begin{multline}
    2\sum_{i, j,k,l = 1}^d (x_j - (x_0)_j) a_{kj}  \frac{\partial^2 v}{\partial x_k \partial x_i}a_{il} \frac{\partial v}{\partial l}
   = \sum_{i, j,k,l = 1}^d (x_j - (x_0)_j) a_{kj}  \frac{\partial^2 v}{\partial x_k \partial x_i}a_{il} \frac{\partial v}{\partial l}\\
   + \sum_{i, j,k,l = 1}^d (x_j - (x_0)_j) a_{kj}  \frac{\partial^2 v}{\partial x_k \partial x_l}a_{li} \frac{\partial v}{\partial i}
   \label{2.1717}
\end{multline}
for all $\x \in \Omega.$
Since $a_{il} = a_{li},$ it follows from \eqref{2.1717} that
\begin{align}
     2\sum_{i, j,k,l = 1}^d &(x_j - (x_0)_j) a_{kj}  \frac{\partial^2 v}{\partial x_k \partial x_i}a_{il} \frac{\partial v}{\partial l} 
    = 
    \sum_{i, j,k,l = 1}^d (x_j - (x_0)_j) a_{kj} a_{il} \Big[  \frac{\partial^2 v}{\partial x_k \partial x_i} \frac{\partial v}{\partial l}
    + \frac{\partial^2 v}{\partial x_k \partial x_l} \frac{\partial v}{\partial i}
    \Big] \notag
    \\
    &= \sum_{i, j,k,l = 1}^d (x_j - (x_0)_j) a_{kj} a_{il} \frac{\partial }{\partial x_k}\Big(
        \frac{\partial v}{\partial x_i} \frac{\partial v}{\partial x_l}
    \Big) \notag
    \\
    &=
    \sum_{i, j,k,l = 1}^d \frac{\partial }{\partial x_k}\Big( (x_j - (x_0)_j) a_{kj} a_{il} 
        \frac{\partial v}{\partial x_i} \frac{\partial v}{\partial x_l}
    \Big) 
- \sum_{i, j,k, l = 1}^d \frac{\partial}{\partial x_k}\Big((x_j - (x_0)_j) a_{kj} a_{il} \Big)
        \frac{\partial v}{\partial x_i} \frac{\partial v}{\partial x_l}.
    \label{c_2_2.12p}
\end{align}
The first sum of in the right hand side of \eqref{c_2_2.12p} can be rewritten as 
\begin{equation}
    \sum_{i, j,k,l = 1}^d \frac{\partial }{\partial x_k}\Big( (x_j - (x_0)_j) a_{kj} a_{il} 
        \frac{\partial v}{\partial x_i} \frac{\partial v}{\partial x_l} \Big)
        =
        \sum_{k, j = 1}^d  \frac{\partial }{\partial x_k}\Big( (x_j - (x_0)_j) a_{kj} \sum_{i,l = 1}^d a_{il} 
        \frac{\partial v}{\partial x_i} \frac{\partial v}{\partial x_l}\Big)
        = \Div (V_2)
        \label{2.1919}
\end{equation}
where
\begin{equation}
    V_2 = A (\x - \x_0) (A \nabla v \cdot \nabla v).
    \label{V2}
\end{equation}
By \eqref{c_2_2.12p}, \eqref{2.1919} and \eqref{V2}, we have proved that
\begin{equation}
     2\sum_{i, j,k,l = 1}^d (x_j - (x_0)_j) a_{kj}  \frac{\partial^2 v}{\partial x_k \partial x_i}a_{il} \frac{\partial v}{\partial l}
     = \Div(V_2) -  \sum_{i, j,k, l = 1}^d \frac{\partial}{\partial x_k}\Big((x_j - (x_0)_j) a_{kj} a_{il} \Big)
        \frac{\partial v}{\partial x_i} \frac{\partial v}{\partial x_l}.
        \label{c2_2.16}
\end{equation}
We are now at the position to estimate $I_1$.  Using \eqref{c_2_2.9}, \eqref{V1}, \eqref{c_2_2.11}, \eqref{V2} and \eqref{c2_2.16}, we obtain
\begin{equation}
    I_1 \geq \Div(V_1 + V_2) - C_1 |\nabla v|^2
    \label{2.23}
\end{equation}
 where 
 \[
 C_1 = 2\max_{\x \in \overline \Omega}\Big\{\Big|\frac{\partial}{\partial x_k}\Big((x_j - (x_0)_j) a_{kj} a_{il}\Big) \Big|\Big\}
 \] is a constant depending only on $A$, $\x_0$, $\Omega$, and $d$. 

\noindent {\bf Step 3.} We now estimate $I_2.$
A simple computation yields
\begin{align*}
    \Div(A\nabla e^{-\lambda r^{-\beta}})   
    &= \lambda \beta  \Div(r^{-\beta - 2} e^{-\lambda r^{-\beta}}  A (\x - \x_0) )\\
    &= \lambda \beta \Big[
        \nabla (r^{-\beta - 2} e^{-\lambda r^{-\beta}}) \cdot A (\x - \x_0) 
    + r^{-\beta - 2} e^{-\lambda r^{-\beta}} \Div(A(\x - \x_0))
    \Big]
\end{align*}
for all $\x \in \Omega$. Thus,
\begin{multline}
    \Div(A\nabla e^{-\lambda r^{-\beta}})
    = 
    \lambda \beta  e^{-\lambda r^{-\beta}} \Big[
        \big( 
            -(\beta + 2) r^{-\beta-4} + \lambda \beta r^{-2\beta -4} 
        \big) (\x - \x_0)
        \cdot A (\x - \x_0) 
        \\
     +   r^{-\beta - 2} \Div(A(\x - \x_0))\Big]
     \label{c2_2.18}
\end{multline}
for all $\x \in \Omega.$
Since $A$ is symmetric, recalling \eqref{I2} and using \eqref{c2_2.18}, we can write
\[
    I_2 = e^{\lambda r^{-\beta}} A(\x - \x_0) \cdot \nabla (|v|^2) 
      \Div (A\nabla e^{\lambda r^{-\beta}}).
\]
Hence,
\begin{equation*}
    I_2 = \lambda \beta  A(\x - \x_0) \cdot \nabla (|v|^2)
       \Big[
        \big( 
            -(\beta + 2) r^{-\beta-4} + \lambda \beta r^{-2\beta -4} 
        \big) (\x - \x_0)
        \cdot A (\x - \x_0) 
    +   r^{-\beta - 2}\Div(A(\x - \x_0))\Big].
\end{equation*}
Thus,
\begin{equation}
    I_2 =  \Div(V_3 ) - \lambda \beta|v|^2 \Div (P)
    \label{2.19}
\end{equation}
where
\begin{equation}
    V_3 = \lambda \beta |v|^2 P 
    \label{V3}
\end{equation}
and
\begin{equation}
    P =     \Big[
        \big( 
            -(\beta + 2) r^{-\beta-4} + \lambda \beta r^{-2\beta -4} 
        \big) (\x - \x_0)
        \cdot A (\x - \x_0) 
     +   r^{-\beta - 2} \Div(A(\x - \x_0))\Big]
       A(\x - \x_0).
       \label{P}
\end{equation}
We estimate the second term in the right hand side of \eqref{2.19}. We write
\begin{equation}
    -\lambda \beta |v|^2 \Div (P) = -\lambda \beta |v|^2 \Div(P1 + P_2 + P_3)
    \label{2.27}
\end{equation}
where
\begin{align*}
    P_1 &= -(\beta + 2)r^{-\beta - 4} (\x - \x_0) \cdot A(\x - \x_0)A(\x - \x_0),\\
    P_2 &= \lambda \beta r^{-2\beta - 4}(\x - \x_0) \cdot A(\x - \x_0)A(\x - \x_0),\\
    P_3 &= r^{-\beta - 2} \Div (A(\x - x_0)) A(\x - \x_0).
\end{align*}
Simple computations yield
\begin{align}
   - \Div (P_1) &= (\beta + 2)\Big[\nabla (r^{-\beta - 4}) \cdot[ (\x - \x_0) \cdot A(\x - \x_0)A(\x - \x_0)] \notag
   \\
   &\hspace{1cm}+ r^{-\beta - 4}\Div((\x - \x_0) \cdot A(\x - \x_0)A(\x - \x_0)) \notag
   \\
   &= (\beta + 2)\Big[-(\beta + 4) r^{-\beta - 6} (\x - \x_0) \cdot[ (\x - \x_0) \cdot A(\x - \x_0)A(\x - \x_0)] \notag
  \\
  &\hspace{1cm}
   + r^{-\beta - 4}\Div((\x - \x_0) \cdot A(\x - \x_0)A(\x - \x_0))
    \Big].
    \label{2.2929}
\end{align}
Recalling \eqref{ellipic_matrix}, we have
\begin{equation}
    (\x - \x_0)\cdot A(\x - \x_0) \geq \Lambda^{-1}
|\x - \x_0|^2 = \Lambda^{-1} r^2.
\label{2.3030}
\end{equation}
It follows from \eqref{2.2929} and \eqref{2.3030} that
\begin{equation}
    - \Div (P_1) \geq -(\beta + 2)(\beta + 4) \Lambda^{-1} r^{-\beta - 2} - C_2 r^{-\beta - 2} \geq -C_3 \beta^2 r^{-\beta - 2}
    \label{2.28}
\end{equation}
where \[C_2 = \max_{\x \in \overline \Omega}\Big\{r^{-2}|\Div((\x - \x_0) \cdot A(\x - \x_0)A(\x - \x_0))|\Big\}\]
and $C_3$ depends only on $\Omega$, $\x_0,$ and $\Lambda$.
We next estimate $-\Div(P_2).$ We have
\begin{align*}
    -\Div(P_2) &= - \lambda \beta\Big[ \nabla (r^{-2\beta - 4})\cdot [(\x - \x_0) \cdot A(\x - \x_0)A(\x - \x_0)]
    \\
    &\hspace{2cm}+ 
    r^{-2\beta - 4} \Div[(\x - \x_0) \cdot A(\x - \x_0)A(\x - \x_0)]
    \Big]
    \\
    &= -\lambda \beta \Big[ 
        (-2\beta - 4) r^{-2\beta - 6} [(\x - \x_0)\cdot A(\x - \x_0)]^2
        \\
        &\hspace{2cm}+ 
    r^{-2\beta - 4} \Div[(\x - \x_0) \cdot A(\x - \x_0)A(\x - \x_0)]
    \Big].
\end{align*}
Using \eqref{2.3030}, we have
\begin{equation}
    -\Div(P_2) \geq C_4\lambda \beta^2 r^{-2\beta - 2}
    \label{2.29}
\end{equation}
where $C_4$ depends only on $\Omega$, $\x_0,$ and $\Lambda$.
On the other hand,
\begin{align*}
    -\Div(P_3) &= -\Big[\nabla (r^{-\beta - 2})\cdot[ \Div(A(\x - \x_0)) A(\x - \x_0)] 
    \\
    &\hspace{3cm}
    + r^{-\beta - 2} \Div(\Div(A(\x - \x_0)) A(\x - \x_0))\Big]\\
    &= (\beta + 2) r^{-\beta - 4} (\x - \x_0)\cdot[ \Div(A(\x - \x_0)) A(\x - \x_0)] 
    \\
    &\hspace{3cm}
    - r^{-\beta - 2} \Div(\Div(A(\x - \x_0)) A(\x - \x_0)).
\end{align*}
Hence,
\begin{equation}
    -\Div(P_3) \geq -C_5 \beta r^{-\beta - 2},
    \label{2.30}
\end{equation}
where $C_5$ depends only on $\Omega$, $\x_0,$ and $\Lambda$.
Combining \eqref{2.27}, \eqref{2.28}, \eqref{2.29} and \eqref{2.30}, we have
\begin{equation}
    -\lambda \beta |v|^2 \Div(P) \geq C_6\lambda^2 \beta^3 r^{-2\beta - 2} |v|^2,
    \label{2.31}
\end{equation}
where $C_6$ depends only on $\Omega$, $\x_0,$ and $\Lambda$.
Here, we have used the fact that $\lambda R^{-\beta} \geq 2.$
Due to \eqref{2.19} and \eqref{2.31}, we obtain
\begin{equation}
    I_2 \geq \Div(V_3)  - C_6\lambda^2 \beta^3 r^{-2\beta - 2} |v|^2.
    \label{2.32}
\end{equation}
Step 3 is complete. 

\noindent {\bf Step 4.}
 Combining the estimates \eqref{2.141414},
\eqref{2.23} and \eqref{2.32}, we get
\begin{equation}
    \frac{r^{\beta + 2} e^{2\lambda r^{-\beta}}|\Div (A\nabla u)|^2}{2\lambda \beta} 
    \\
    \geq \Div(V_1 + V_2 + V_3) 
    + C_6\lambda^2 \beta^3 r^{-2\beta - 2}|v|^2 - C_1 |\nabla v|^2 
    \label{2.34}
\end{equation}
for all $\x \in \Omega.$
Recall from \eqref{2.10} that $v = e^{\lambda r^{\beta}} u$. By standard rule in differentiation, we have
\begin{equation*}
    \nabla v = e^{\lambda r^{-\beta}}(-\lambda \beta u r^{-\beta - 2} (\x - \x_0)  + \nabla u).
\end{equation*}
Hence,
\begin{equation}
    |\nabla v|^2 \geq -C_7 e^{2\lambda r^{-\beta}} (\lambda^2 \beta^2 r^{-2\beta - 2}|u|^2 + |\nabla u|^2),
    \label{2.35}
\end{equation}
where $C_6$ depends only on $\Omega$ and $\x_0.$
Combining \eqref{2.10}, \eqref{2.34} and \eqref{2.35}, we obtain 
\begin{equation}
    \frac{r^{\beta + 2} e^{2\lambda r^{-\beta}}|\Div (A\nabla u)|^2}{2\lambda \beta} \geq \Div(V_1 + V_2 + V_3)
    + C_8\lambda^2 \beta^3 e^{2\lambda r^{-\beta}} r^{-2\beta - 2}|u|^2 - C_9 e^{2\lambda r^{-\beta}} |\nabla u|^2 
    \label{2.3838}
\end{equation}
for all $\x \in \Omega,$
where $C_8$ and $C_9$ depend only on $\Omega$, $\x_0,$ and $\Lambda$.

\noindent {\bf Step 5.} 
We have
\begin{align}
    e^{2\lambda r^{-\beta}} u \Div(A\nabla u) 
    &=  \Div( e^{2\lambda r^{-\beta}} u A\nabla u) - \nabla [e^{2\lambda r^{-\beta}} u] \cdot (A\nabla u)
    \notag\\
    &= \Div( U_1) - e^{2\lambda r^{-\beta}}\Big[
        \nabla u - 2\lambda \beta r^{-\beta - 2}u (\x - \x_0)\Big] A\nabla u \label{2.38}
\end{align}
where 
\begin{equation}
    U_1 = e^{2\lambda r^{-\beta}} u A\nabla u. 
    \label{U1}
\end{equation}
Since
\[
    2\lambda \beta r^{-\beta - 2} u A\nabla u \cdot (\x - \x_0) \leq \frac{1}{2\Lambda} |\nabla u|^2 + 8C_{10} \lambda^2\beta^2 \Lambda r^{-2\beta - 2} |u|^2,
\]
using \eqref{ellipic_matrix} with $\xi = \nabla u$, \eqref{2.38} and the inequality $2ab \leq a^2 + b^2$, we have
\begin{equation}
    e^{2\lambda r^{-\beta}} u \Div(A\nabla u)  
    \leq \Div(U_1) + C_{11}\lambda^2\beta^2 e^{2\lambda r^{-\beta}} r^{-2\beta - 2}|u|^2 - \frac{1}{2\Lambda} e^{2\lambda r^{-\beta}} |\nabla u|^2. 
\end{equation}
Here, $C_{10}$ and $C_{11}$ depend only on $\Omega$, $\x_0,$ and $\Lambda$.

On the other hand,
since
\[
|u\Div (A\nabla u)| \leq \lambda^2 \beta |u|^2 r^{-2\beta - 2} + \frac{4}{\lambda \beta} |\Div(A\nabla u)|^2 r^{\beta + 2},
\]
we have
\begin{align}
    \frac{4r^{\beta + 2} e^{2\lambda r^{-\beta}} |\Div(A\nabla u)|^2}{\lambda \beta} 
    &\geq e^{2\lambda r^{-\beta}} |u \Div(A\nabla u)| - \lambda^2 \beta r^{-2\beta - 2}e^{2\lambda r^{-\beta}}|u|^2 \notag\\
     &\geq -\Div (U_1) - C_{12}\lambda^2 \beta^2 e^{2\lambda r^{-\beta}} r^{-2\beta - 2} |u|^2 
    + \frac{1}{2\Lambda}e^{2\lambda r^{-\beta}}|\nabla u|^2,
    \label{2.42}
\end{align}
where $C_{12}$ depend only on $\Omega$, $\x_0,$ and $\Lambda$.
Multiply both sides of \eqref{2.42} by $4C_1 \Lambda$ and then add the resulting equation into \eqref{2.3838}. We obtain
\begin{equation}
    \frac{r^{\beta + 2} e^{2\lambda r^{-\beta}} |\Div(A\nabla u)|^2}{\lambda \beta} 
    \geq C\Big[
        \Div (U_2) + \lambda^2\beta^3e^{2\lambda r^{-\beta}} r^{-2\beta - 2} |u|^2 + e^{2\lambda r^{-\beta}} |\nabla u|^2 
    \Big]
    \label{2.43}
\end{equation}
    where
    \[
        U_2= -U_1 + V_1 + V_2 + V_3.
    \]
    Due to \eqref{V1}, \eqref{V2}, \eqref{V3}, \eqref{P} and \eqref{U1}, it is obvious that 
    \[
        |U_2| \leq Ce^{2\lambda r^{-\beta}}(\lambda^2 \beta^2 r^{-2\beta - 2}|u|^2 + |\nabla u|^2).
    \]
    Letting $U = \lambda \beta U_2$, we obtain \eqref{CarEst}.
The proof is complete.
\end{proof}

\begin{Corollary}
Fix $\beta \geq \beta_0$. There exists a number $\lambda_0$ depending only on $\Lambda$, $\|A\|_{C^2(\overline \Omega)}$, $\x_0$, $\Omega$, $R$, $\beta$ and $d$ such that for all $\lambda \geq \lambda_0$,
\begin{equation}
         e^{2\lambda r^{-\beta}} |\Div(A\nabla u)|^2  
    \geq C\Big[
        \Div (U) + \lambda^3 e^{2\lambda r^{-\beta}}  |u|^2 + \lambda  e^{2\lambda r^{-\beta}} |\nabla u|^2 
    \Big]
    \label{CarEst1}
    \end{equation}
    where $C$ is a constant depending only on $\Lambda$, $\|A\|_{C^2(\overline \Omega)}$, $\x_0$, $\Omega$, $R$, $\beta$ and $d$.
    \end{Corollary}
    \begin{Corollary}
    Integrating \eqref{CarEst1} on $\Omega$ and using  \eqref{divU}, we obtain 
    \begin{equation}
        \int_{\Omega} e^{2\lambda r^{-\beta}} |\Div A\nabla u|^2d\x
        \geq 
        C \int_{\Omega} e^{2\lambda r^{-\beta}}\big[
            \lambda^3 |u|^2 + \lambda |\nabla u|^2
        \big]d\x
        - C\int_{\partial \Omega}e^{2\lambda r^{-\beta}} \big[\lambda^3|u|^2 + \lambda |\nabla u|^2\big]d\sigma(\x).
        \label{3.40}
    \end{equation}
    In particular, if $u$ is a function that satisfies $u|_{\partial \Omega} = 0$ and
     $\nabla u|_{\partial \Omega}= 0$. Then, \begin{equation}
        \int_{\Omega} e^{2\lambda r^{-\beta}} |\Div A\nabla u|^2d\x
        \geq 
        C \int_{\Omega} e^{2\lambda r^{-\beta}}\big[
            \lambda^3 |u|^2 + \lambda |\nabla u|^2
        \big]d\x.
        \label{CarlemanExercise}
    \end{equation}
    Here, $C$ is a constant depending only on $\Lambda$, $\|A\|_{C^2(\overline \Omega)}$, $\x_0$, $\Omega$, $R$, $\beta$ and $d$.
    \label{col222}
\end{Corollary}

\begin{Remark}
The Carleman estimate in \eqref{CarlemanExercise} is similar to  \cite[Lemma 5]{MinhLoc:tams2015}. The main difference is that the result in  \cite[Lemma 5]{MinhLoc:tams2015} is for annulus domains while estimate \eqref{CarlemanExercise} is applicable for more general domains.
It is interesting mentioning that the Carleman estimate in  \cite[Lemma 5]{MinhLoc:tams2015} for annulus domains was used to prove a cloaking phenomenon, see \cite{MinhLoc:tams2015}.
The reader can find many other versions of Carleman estimates in \cite{BeilinaKlibanovBook, KlibanovLiBook, KlibanovLeNguyenIPI2021, LocNguyen:ip2019, Protter:1960AMS}.
These estimates are used to solve inverse problems; see e.g., \cite{KhoaKlibanovLoc:SIAMImaging2020, LeNguyen:jiip2022, NguyenNguyenTruong:arxiv2022}.
\end{Remark}


\section{The Carleman convexification theorem}\label{sec_convexi}

Let $p > \ceil{d/2} + 2.$ 
We have $H^p(\Omega)$ is continuously embedded into $C^2(\overline \Omega).$
Fix $\beta = \beta_0$. For all $\lambda > \lambda_0$ and for $\eta \in (0, 1),$ define the Carleman weighted mismatch functional $J_{\lambda, \eta}: H^p(\Omega) \to \R$ as follows
\begin{equation}
    J_{\lambda, \eta}(v) = 
    \int_{\Omega} e^{2\lambda r^{-\beta}} |-\epsilon_0 \Delta v + F(\x, v, \nabla v)|^2 d\x
    + \lambda^4\int_{\partial \Omega}  e^{2\lambda r^{-\beta}} (|v|^2 + |\nabla v|^2) d\sigma(\x)
    + \eta \|v\|^2_{H^p(\Omega)}.
    \label{4.1}
\end{equation}

    The Carleman weighted mismatch functional $J_{\lambda, \eta}$ in \eqref{4.1} is different from the ones used in our research group's previous papers \cite{KhoaKlibanovLoc:SIAMImaging2020, KlibanovNguyenTran:JCP2022, LeNguyen:JSC2022}. The main differences is that in \eqref{4.1}, we include the integral on $\partial \Omega$. 
    We add this boundary integral to the mismatch functional because we do not know the exact boundary information of the function $v_{\epsilon_0}$ on $\partial \Omega.$
    The presence of this boundary integral somewhat guarantees that the values of $v^{\epsilon_0}_{\rm min}|_{\partial \Omega}$ and $\nabla v^{\epsilon_0}_{\rm min}|_{\partial \Omega}$ are small where $v^{\epsilon_0}_{\rm min}$ is the minimizer of $J_{\lambda, \eta}$.
    Also, since we will
     minimize $J_{\lambda, \eta}$  without  boundary constraints,
     the earlier versions of the Carleman convexification method \cite{KlibanovNik:ra2017, KhoaKlibanovLoc:SIAMImaging2020, KlibanovNguyenTran:JCP2022, LeNguyen:JSC2022}, which require some boundary conditions on the minimizer,  are not applicable. 
     We modify the use of the Carleman estimate in those theorem to obtain the convexification theorem below.

\begin{Theorem}[The convexification theorem] Assume that the function $F$ is of class $C^2(\R^d \times \R \times \R^d).$ We have:

\begin{enumerate}
    \item 
 For all $\lambda > 1$ and $\eta > 0$, 
the functional $J_{\lambda, \nu}$ is Fr\'etchet differentiable. The derivative of $J_{\lambda, \eta}$ is given by
\begin{multline}
    DJ_{\lambda, \eta}(v)h =
    2\int_{\Omega}e^{2\lambda r^{-\beta}}[-\epsilon_0 \Delta v + F(\x, v, \nabla v)][-\epsilon_0 \Delta h + \partial_{s} F(\x, v, \nabla v)h
    + 
    \nabla_{\p}F(\x, v, \nabla v)\cdot \nabla h] d\x
    \\
    +
    2\lambda^4 \int_{\partial \Omega} e^{2\lambda r^{-\beta}} [v h +  
     \nabla v \cdot \nabla h] d\sigma(\x) 
    +2\eta \langle v, h\rangle_{H^p(\Omega)}
\end{multline}
for all $v, h \in H^p(\Omega).$
Here, $\partial_s F$ is the partial differential derivative of the function $F(\x, s, \p)$, $(\x, s, \p) \in \Omega \times \R \times \R^{d}$, with respect to the second variable and $\nabla_{\p} F$ is the gradient vector of $F$ with respect to the third variable $\p.$

\item	Let $M$ be an arbitrarily large number. 
For each $\beta> 1$,
  $\lambda >\lambda_0 = \lambda_0(\epsilon_0, M, b, d, r, F, \beta) > 1$, $\eta > 0$, $u, v \in   \overline{B(M)}$, we have
  \begin{multline}
        J_{\lambda, \eta}(u) - J_{\lambda, \eta}(v) - D J_{\lambda, \eta}(v)(u - v)
        \geq C\int_{\Omega}e^{2\lambda r^{-\beta}} \big[
             |u - v|^2 +  |\nabla (u - v)|^2
        \big]d\x
        \\
        + C\int_{\partial \Omega} \lambda^4 e^{2\lambda r^{-\beta}} (|u - v|^2 + |\nabla (u - v)|^2)d\sigma(\x)
      + \eta \|u - v\|_{H^p(\Omega)}^2
      \label{convexity}
    \end{multline}
Here, the constant $C$ depends only on $\lambda,$ $\beta$, $R,$ $r,$ $d$, $M$, $F$ and $\epsilon_0$.

\item The functional $J_{\lambda, \beta, \eta}$ has a unique minimizer in $\overline{B(M)}$.
\end{enumerate}
\label{thm convex}
\end{Theorem}

\begin{Remark}
    An intuition for the convexity of $J_{\lambda, \eta}$ is that one can apply the convexification theorem in \cite{KlibanovNguyenTran:JCP2022} to obtain the convexity of the functional
    \[
        v \mapsto  
    I_{\lambda, \eta}(v) = \int_{\Omega} e^{2\lambda r^{-\beta}} |-\epsilon_0 \Delta v + F(\x, v, \nabla v)|^2 d\x
    + \eta \|v\|^2_{H^p(\Omega)}.
    \]
    By adding the convex term  $ \ds\lambda^4\int_{\partial \Omega} e^{2\lambda r^{-\beta}} [|v|^2d\sigma(\x)
    +   |\nabla v|^2]d\sigma(\x)$ to this functional, we obtain the desired convexity of $J_{\lambda, \eta}$.
    However, the convexity of $I_{\lambda, \eta}$ is valid only on a set of functions that satisfy some Cauchy boundary data. 
    Hence, the informal argument above is not rigorous.
  We present the proof of Theorem \ref{thm convex} here.
\end{Remark}

\begin{proof}[Proof of Theorem \ref{thm convex}]
The first part of Theorem \ref{thm convex} can be proved by a straight forward computation similarly in the first part of \cite[Theorem 4.1]{KlibanovNguyenTran:JCP2022}. We now  discuss part 2 of Theorem \ref{thm convex}.
Let $u$ and $v$ be two functions in $H^p(\Omega).$  Let $h = u - v$.
We have
\begin{multline}
    J_{\lambda, \eta}(u) - J_{\lambda, \eta}(v) - D J_{\lambda, \eta}(v)(u - v) 
    \\
    = 
    \int_{\Omega} e^{2\lambda r^{-\beta}}\big[ |-\epsilon_0 \Delta u + F(\x, u, \nabla u)|^2 -  |-\epsilon_0 \Delta v + F(\x, v, \nabla v)|^2 \big]d\x
    \\
    + \lambda^4\int_{\partial \Omega} e^{2\lambda r^{-\beta}} \big[u^2 - v^2 + |\nabla u|^2 - |\nabla v|^2\big]d\sigma(\x)
    + \eta \big[\|u\|^2_{H^p(\Omega)} - \|v\|^2_{H^p(\Omega)}\big]
    \\
    -2\int_{\Omega}e^{2\lambda r^{-\beta}}[-\epsilon_0 \Delta v + F(\x, v, \nabla v)][-\epsilon_0 \Delta h + \partial_{s} F(\x, v, \nabla v)h
    + 
    \nabla_{\p}F(\x, v, \nabla v)\cdot \nabla h] d\x
    \\
    -2\lambda^4 \int_{\partial \Omega} e^{2\lambda r^{-\beta}} [v h +  
     \nabla v \cdot \nabla h] d\sigma(\x) 
    -2\eta \langle v, h\rangle_{H^p(\Omega)}.
    \label{4.5}
\end{multline}
Using the identity $a^2 - b^2 = (a - b)(a +b)$, we deduce from \eqref{4.5} that
\begin{multline}
    J_{\lambda, \eta}(u) - J_{\lambda, \eta}(v) - D J_{\lambda, \eta}(u - v) 
    = 
    \int_{\Omega} e^{2\lambda r^{-\beta}}\big[ -\epsilon_0 \Delta h +F(\x, u, \nabla u) -    F(\x, v, \nabla v)
    \\-2\epsilon_0 \Delta v + 2F(\x, v, \nabla v)
     \big]
    \big[ -\epsilon_0 \Delta h + F(\x, u, \nabla u) -    F(\x, v, \nabla v) \big]
    d\x
    \\
    + \lambda^4\int_{\partial \Omega} e^{2\lambda r^{-\beta}} [(u + v)h + 
    \nabla (u +  v) \cdot \nabla h] d\sigma(\x)
    + \eta \langle u + v, h\rangle_{H^p(\Omega)}
    \\
    -2\int_{\Omega}e^{2\lambda r^{-\beta}}[-\epsilon_0 \Delta v + F(\x, v, \nabla v)][-\epsilon_0 \Delta h + \partial_{s} F(\x, v, \nabla v)h
    + 
    \nabla_{\p}F(\x, v, \nabla v)\cdot \nabla h] d\x
    \\
    -2\lambda^4 \int_{\partial \Omega} e^{2\lambda r^{-\beta}} [v h +  
     \nabla v \cdot \nabla h] d\sigma(\x) 
    -2\eta \langle v, h\rangle_{H^p(\Omega)}.
    \label{4.6}
\end{multline}
Expending the right hand side of \eqref{4.6}, we have
 \begin{equation}
      J_{\lambda, \eta}(u) - J_{\lambda, \eta}(v) - D J_{\lambda, \eta}(u - v) = I_1 + I_2 + I_3 + \lambda^4 \int_{\partial \Omega} e^{2\lambda r^{-\beta}} (|h|^2 + |\nabla h|^2)d\sigma(\x)
      + \eta \|h\|_{H^p(\Omega)}^2
      \label{4.7}
 \end{equation}
 where 
 \begin{align*}
     I_1 &= \int_{\Omega} e^{2\lambda r^{-\beta}}\big| -\epsilon_0 \Delta h +F(\x, u, \nabla u) -    F(\x, v, \nabla v)\big|^2d\x,
     \\
     I_2 &= 2\int_{\Omega} e^{2\lambda r^{-\beta}}
     \big[
     -\epsilon_0 \Delta v + F(\x, v, \nabla v)
     \big]
     \big[ -\epsilon_0 \Delta h +F(\x, u, \nabla u) -    F(\x, v, \nabla v)\big]d\x,
     \\
     I_3 &= -2\int_{\Omega}e^{2\lambda r^{-\beta}}[-\epsilon_0 \Delta v + F(\x, v, \nabla v)][-\epsilon_0 \Delta h + \partial_{s} F(\x, v, \nabla v)h
    + 
    \nabla_{\p}F(\x, v, \nabla v)\cdot \nabla h] d\x.
 \end{align*}
 Using the inequality $(a - b)^2 \geq \frac{1}{2}a^2 - b^2$ and recalling that $u$ and $v$ are in the bounded set $\overline{B(M)}$, we can find a constant $C$ such that
 \begin{align}
     I_1 &\geq \frac{\epsilon_0^2}{2}\int_{\Omega} e^{2\lambda r^{-\beta}} |\Delta h|^2d\x 
     - 
     \int_{\Omega} e^{2\lambda r^{-\beta}} |F(\x, u, \nabla u) - F(\x, v, \nabla v)|^2d\x
     \notag\\
     &\geq 
     \frac{\epsilon_0^2}{2}\int_{\Omega} e^{2\lambda r^{-\beta}} |\Delta h|^2d\x 
     - 
     C\int_{\Omega} e^{2\lambda r^{-\beta}} [|h|^2 + |\nabla h|^2]d\x.
     \label{4.8}
 \end{align}
 On the other hand,
 \begin{multline}
     I_2 + I_3 = - 2\int_{\Omega} e^{2\lambda r^{-\beta}}\big[
     -\epsilon_0 \Delta v + F(\x, v, \nabla v)
     \big]
     \\
     \times
    \big[F(\x, u, \nabla u) - F(\x, v, \nabla v) + \partial_{s} F(\x, v, \nabla v)h
    + 
    \nabla_{\p}F(\x, v, \nabla v)\cdot \nabla h\big]d\x
    \label{4.9}
 \end{multline}
     Since both $u$ and $v$ are in the set $\overline{B(M)}$, we have
     \[
    \big| F(\x, u, \nabla u) - F(\x, v, \nabla v) + \partial_{s} F(\x, v, \nabla v)h
    + 
    \nabla_{\p}F(\x, v, \nabla v)\cdot \nabla h\big| \leq C[|h|^2 + |\nabla h|^2].
     \]
     Thus,
     \begin{equation}
         I_2 + I_3 \geq -C\int_{\Omega} e^{2\lambda r^{-\beta}} [|h|^2 + |\nabla h|^2]d\x.
    \label{4.10}
     \end{equation}
 Combining \eqref{4.7}, \eqref{4.8}, \eqref{4.9} and \eqref{4.10}, we have
 \begin{multline}
     J_{\lambda, \eta}(u) - J_{\lambda, \eta}(v) - D J_{\lambda, \eta}(u - v) 
     \geq 
     \frac{\epsilon_0^2}{2}\int_{\Omega} e^{2\lambda r^{-\beta}} |\Delta h|^2d\x
     - C\int_{\Omega} e^{2\lambda r^{-\beta}} [|h|^2 + |\nabla h|^2]d\x
     \\
     + \lambda^4 \int_{\partial \Omega} e^{2\lambda r^{-\beta}} (|h|^2 + |\nabla h|^2)d\sigma(\x)
      + \eta \|h\|_{H^p(\Omega)}^2.
      \label{4.11}
 \end{multline}
 
 In order to prove the convexity of $J_{\lambda, \eta}$, we need to show that the right hand side \eqref{4.11} is nonnegative.
 This is the main reason why the Carleman estimate in \eqref{3.40} plays the key role in this proof.
 Applying \eqref{3.40} for the function $h$ with $A = \Id$, we have
 \begin{equation}
        \frac{\epsilon_0^2}{2}\int_{\Omega} e^{2\lambda r^{-\beta}} |\Delta h|^2d\x
        \geq 
        \frac{C\epsilon_0^2}{2} \int_{\Omega} e^{2\lambda r^{-\beta}}\big[
            \lambda^3 |h|^2 + \lambda |\nabla h|^2
        \big]d\x
        - \frac{C\epsilon_0^2}{2}\int_{\partial \Omega}e^{2\lambda r^{-\beta}} \big[\lambda^3|h|^2 + \lambda |\nabla h|^2\big]d\sigma(\x).
        \label{4.12}
    \end{equation}
     Letting $\lambda$ be sufficiently large, allowing $C$ to depend on $\epsilon_0$ and $\lambda$, combining \eqref{4.11} and \eqref{4.12}, and recalling that $h = u - v$, we get \eqref{convexity}.

     We next show that $J_{\lambda, \eta}$ has a unique minimizer by using the arguments in \cite{KlibanovNik:ra2017}. Assume $J_{\lambda, \eta}$ has two minimizers $v_1$ and $v_2$ in $\overline{B(M)}$.
     Applying \eqref{convexity} for $u = v_1$ and $v = v_2$, we have
     \begin{multline}
        J_{\lambda, \eta}(v_1) - J_{\lambda, \eta}(v_2) - D J_{\lambda, \eta}(v_2)(v_1 - v_2)
        \geq C\int_{\Omega}e^{2\lambda r^{-\beta}} \big[
             |v_1 - v_2|^2 +  |\nabla (v_1 - v_2)|^2
        \big]d\x
        \\
        + C\int_{\partial \Omega} \lambda^4 e^{2\lambda r^{-\beta}} (|v_1 - v_2|^2 + |\nabla (v_1 - v_2)|^2)d\sigma(\x)
      + \eta \|v_1 - v_2\|_{H^p(\Omega)}^2
      \label{4,12}
    \end{multline}
    By \cite[Lemma 2]{KlibanovNik:ra2017}, since $v_2$ is a minimizer of $J_{\lambda, \eta}$ in $\overline{B(M)},$
    \begin{equation}
        DJ_{\lambda, \eta}(v_2)(v_1 - v_2) \geq 0, 
        \quad \mbox{or } \quad
        -DJ_{\lambda, \eta}(v_2)(v_1 - v_2) \leq 0
    \label{4,13}
    \end{equation}
    Combining \eqref{4,12} and \eqref{4,13}, we have
    \begin{multline}
        J_{\lambda, \eta}(v_1) - J_{\lambda, \eta}(v_2)
        \geq 
        C\int_{\Omega}e^{2\lambda r^{-\beta}} \big[
             |v_1 - v_2|^2 +  |\nabla (v_1 - v_2)|^2
        \big]d\x
        \\
        + C\int_{\partial \Omega} \lambda^4 e^{2\lambda r^{-\beta}} (|v_1 - v_2|^2 + |\nabla (v_1 - v_2)|^2)d\sigma(\x)
      + \eta \|v_1 - v_2\|_{H^p(\Omega)}^2.
      \label{4,14}
    \end{multline}
    Similarly, interchanging the roles of $v_1$ and $v_2$, we have
    \begin{multline}
        J_{\lambda, \eta}(v_2) - J_{\lambda, \eta}(v_1)
        \geq 
        C\int_{\Omega}e^{2\lambda r^{-\beta}} \big[
             |v_1 - v_2|^2 +  |\nabla (v_1 - v_2)|^2
        \big]d\x
        \\
        + C\int_{\partial \Omega} \lambda^4 e^{2\lambda r^{-\beta}} (|v_1 - v_2|^2 + |\nabla (v_1 - v_2)|^2)d\sigma(\x)
      + \eta \|v_1 - v_2\|_{H^p(\Omega)}^2.
      \label{4,15}
    \end{multline}
    Adding \eqref{4,14} and \eqref{4,15}, we obtain $v_1 = v_2.$
\end{proof}

The unique minimizer of $J_{\lambda, \eta}$ can be obtained by the the conventional gradient descent method. We refer the reader to \cite[Theorem 2]{LeNguyen:JSC2022} and \cite[Theorem 4.2]{KlibanovNguyenTran:JCP2022} for this fact. Let $v_{\rm min}^{\epsilon_0}$ be the minimizer of $J_{\lambda, \eta}$, one can repeat the proof in \cite{LeNguyen:JSC2022}. 
We next estimate  the distance of minimizer $v_{\rm min}^{\epsilon_0}$ and $v_{\epsilon_0}.$
We have the theorem.
\begin{Theorem}
Let $v_{\epsilon_0}$ be a function satisfying \eqref{2.9} and \eqref{2.12}.
Assume that $v_{\epsilon_0} \in \overline{B(M)}$ for some large number $M$.
Let $\beta, \lambda$ be such that \eqref{convexity} holds true.
Let $v_{\rm min}^{\epsilon_0}$ be the unique minimizer of
$J_{\lambda, \eta}$ in $\overline{B(M)}$.
We have
\begin{multline}
        \int_{\Omega}e^{2\lambda r^{-\beta}} \big[
             |v_{\epsilon_0} - v_{\rm min}^{\epsilon_0}|^2 +  |\nabla (v_{\epsilon_0} - v_{\rm min}^{\epsilon_0})|^2
        \big]d\x
        + \int_{\partial \Omega} \lambda^4 e^{2\lambda r^{-\beta}} (|v_{\epsilon_0} - v_{\rm min}^{\epsilon_0}|^2 + |\nabla (v_{\epsilon_0} - v_{\rm min}^{\epsilon_0})|^2)d\sigma(\x)
        \\
      + \eta \|v_{\epsilon_0} - v_{\rm min}^{\epsilon_0}\|_{H^p(\Omega)}^2
      \leq 
      C\Big[\lambda^4 \delta^2 \int_{\partial \Omega}  e^{2\lambda r^{-\beta}} d\sigma(\x)
    + \eta \|v_{\epsilon_0}\|^2_{H^p(\Omega)}\Big].
      \label{error_est}
    \end{multline}
    \label{thm4.2}
\end{Theorem}
\begin{proof}
    Applying \eqref{convexity} for $u = v_{\epsilon_0}$ and $v = v_{\rm min}^{\epsilon_0}$, we have
    \begin{multline}
        J_{\lambda, \eta}(v_{\epsilon_0}) - J_{\lambda, \eta}(v_{\rm min}^{\epsilon_0}) - D J_{\lambda, \eta}(v_{\rm min}^{\epsilon_0})(v_{\epsilon_0} - v_{\rm min}^{\epsilon_0})
        \geq C\int_{\Omega}e^{2\lambda r^{-\beta}} \big[
             |v_{\epsilon_0} - v_{\rm min}^{\epsilon_0}|^2 +  |\nabla (v_{\epsilon_0} - v_{\rm min}^{\epsilon_0})|^2
        \big]d\x
        \\
        + C\int_{\partial \Omega} \lambda^4 e^{2\lambda r^{-\beta}} (|v_{\epsilon_0} - v_{\rm min}^{\epsilon_0}|^2 + |\nabla (v_{\epsilon_0} - v_{\rm min}^{\epsilon_0})|^2)d\sigma(\x)
      + \eta \|v_{\epsilon_0} - v_{\rm min}^{\epsilon_0}\|_{H^p(\Omega)}^2
      \label{4.16}
    \end{multline}
     Since $v_{\rm min}^{\epsilon_0}$ is the minimizer of $J_{\lambda, \eta}$ in $\overline{B(M)}$, by   \cite[Lemma 2]{KlibanovNik:ra2017},
    \begin{equation*}
        DJ_{\lambda, \eta}(v_{\rm min}^{\epsilon_0})(v_{\epsilon_0} - v_{\rm min}^{\epsilon_0}) \geq 0, 
        \quad  
        \mbox{or }
        \quad
        -DJ_{\lambda, \eta}(v_{\rm min}^{\epsilon_0})(v_{\epsilon_0} - v_{\rm min}^{\epsilon_0}) \leq 0,
    \end{equation*}
    we have
    \begin{equation}
        - J_{\lambda, \eta}(v_{\rm min}^{\epsilon_0}) - D J_{\lambda, \eta}(v_{\rm min}^{\epsilon_0})(v_{\epsilon_0} - v_{\rm min}^{\epsilon_0}) \leq 0.
        \label{4.17}
    \end{equation}
    On the other hand, recalling that $v_{\epsilon_0}$ satisfies \eqref{2.9} and \eqref{2.12}, we have
    \begin{equation}
    J_{\lambda, \eta}(v_{\epsilon_0}) \leq 
    \lambda^4 C\delta^2 \int_{\partial \Omega}  e^{2\lambda r^{-\beta}} d\sigma(\x)
    + \eta \|v_{\epsilon_0}\|^2_{H^p(\Omega)}.
    \label{4.18}
\end{equation}
    Combining \eqref{4.16}, \eqref{4.17}  and \eqref{4.18} yields \eqref{error_est}.
    The proof is complete.
\end{proof}

\begin{Remark}
Fix $\lambda > \lambda_0$. Since the Carleman weight function $e^{2\lambda r^{-\beta}}$ is bounded from below and above by positive constants, it follows from \eqref{error_est} that
\begin{equation}
    \|v_{\epsilon_0} - v_{\rm min}^{\epsilon_0}\|_{H^1(\Omega)}^2 \leq C\big(\delta^2 + \eta \|v_{\epsilon_0}\|_{H^p(\Omega)}^2\big).
    \label{4.20}
\end{equation}
\end{Remark}

\section{Numerical study} \label{sec5}

The analysis in Section \ref{sec_changevariable}, Theorem \ref{thm convex}, Theorem \ref{thm4.2} and estimate \eqref{4.20} suggest  Algorithm \ref{alg} to compute the solution to \eqref{HJ}.
In this section, we present the implementation and some numerical examples.
Note that in Step \ref{s5} of Algorithm \ref{alg}, we have accepted the well-known vanishing viscosity process for Hamilton-Jacobi equations, that guarantees $v_{\epsilon_0}$ approximates the true viscosity solution to \eqref{HJv}.

\begin{algorithm}[ht]
\caption{\label{alg}The procedure to compute the numerical solution to \eqref{HJ} on a domain $G$}
	\begin{algorithmic}[1]
	\STATE \label{s1} Choose $\Omega = (-R, R)^d \Supset G$. Choose a cut-off function $\chi_{\delta}$ as in \eqref{chi} for some $\delta \in (0, 1)$.
	
	\STATE\label{s2} Choose $\x_0 \in \R^d \setminus \Omega$, $\beta > 0$, $\lambda > 0$. Define a Carleman weight function $e^{2\lambda r^{-\beta}}$ where $r(\x) = |\x - \x_0|$ for all $\x \in \Omega.$
	
    \STATE \label{s3} Choose a viscosity parameter $\epsilon_0$ and a regularization parameter $\eta $, both of which are positive and small. Choose $M > 0$ sufficiently large.
    
	\STATE  \label{s4}
	Define and minimize the functional $J_{\lambda, \eta}$ in $\overline{B(M)}$. 
	The minimizer is denoted by $v_{\rm min}^{\epsilon_0}$. 
	\STATE \label{s5} Set the computed solution to \eqref{HJ} in $G$ by the function $u_{\epsilon_0}(	\x) = v_{\rm min}^{\epsilon_0}(\x)/\chi_{\delta}(\x)$ for all $\x \in G.$
\end{algorithmic}
\end{algorithm}

\subsection{Numerical implementation}

We implement Algorithm \ref{alg} to compute the restriction of solution to \eqref{HJ} on $G = (-0.8, 0.8)^d$ using the finite difference method. We set $\Omega = (-R, R)^d$ where $R = 2$. In this section, for simplicity, we consider two cases $d = 1$ and $d = 2$.
We choose the Gaussian-like function $\chi_{\delta}(\x) = e^{-0.5(|\x|^2)}$ for all $\x \in \R^d$. The function $\chi_{\delta}(\x)$ is less than $\delta = 0.1353$ for all $\x \in \R^d \setminus \Omega.$ The number $c$ in \ref{chi} is $0.7362.$ The choices above include details for Step \ref{s1} of Algorithm \ref{alg}.

The Carleman weight function and other parameters in Step \ref{s2} and Step \ref{s3} of Algorithm \ref{alg} are chosen by a trial and error process. 
We manually try many sets of parameters until we obtain an acceptable solution for a reference test (test 1) below. 
We choose $\x_0 = (9, 0)$, $\beta = 20$, $\lambda = 3$, $\epsilon_0 = 10^{-3}$ and $\eta = 10^{-3}.$
These parameters are used for all other tests.

In Step \ref{s4}, we rewrite the function $J_{\lambda, \eta}$ in the finite difference scheme.
Consider the case when $d = 2.$
Let $N$ be a positive integer. Let $h = 2R/(N - 1)$ represent the step size in space.
On $\Omega$, we arrange a set of $N \times N$ uniform grid points $\Omega_h$ as
\[
    \Omega_{h} = \big\{ \x_{ij} = (x_i = -R + (i-1)h, y_j = -R + (j-1)h), 1 \leq i, j \leq N - 1
    \big\}.
\]
In all of numerical examples below, $N = 70.$
In 2D, the functional $J_{\lambda, \eta}$ is approximated in finite difference as
\begin{multline}
    J_{\lambda, \eta}^h(v) =  h^2\sum_{i, j = 2}^{N-1}  e^{2\lambda r^{-\beta}(\x_{ij})}\big|
    -\epsilon_0 \Delta_h v(\x_{ij}) + F(\x_{ij}, v(\x_{ij}), \nabla_h v(\x_{ij})) 
    \big|^2
    \\
     +h\lambda^4 \sum_{i = 1}^N  e^{2\lambda r^{-\beta}(\x_{i1})} (|v(\x_{i1})|^2 + |\nabla_h v(\x_{i1})|^2) 
     +h\lambda^4\sum_{i = 1}^N  e^{2\lambda r^{-\beta}(\x_{iN})} (|v(\x_{iN})|^2 + |\nabla_h v(\x_{iN})|^2) 
    \\ 
    +h\lambda^4 \sum_{j = 1}^N  e^{2\lambda r^{-\beta}(\x_{1j})} (|v(\x_{1j})|^2 + |\nabla_h v(\x_{1j})|^2) 
     +h\lambda^4 \sum_{j = 1}^N  e^{2\lambda r^{-\beta}(\x_{Nj})} (|v(\x_{Nj})|^2 + |\nabla_h v(\x_{Nj})|^2)
     \\
     + \eta h^2 \sum_{i,j = 2}^{N - 1}\big(v^2(\x_{ij}) + |\nabla_h v(\x_{ij}) + |\Delta_h v(\x_{ij})|^2|\big).
     \label{5.1}
\end{multline}
In \eqref{5.1}, we have reduced the norm in the regularization term to $p = 2$ to simplify the implementation and to improve the speed of computation. We do not experience any difficulty with this small change. In our implementation, instead of writing the computational code for the gradient descent method, we use the optimization toolbox of Matlab, in which the gradient descent method is coded. More precisely, we use the command ``fminunc" to minimize the functional $J_{\lambda, \eta}$. The command ``fminunc" requires an initial solution $v_0$. We choose $v_0 \equiv 0$ in all tests.  
Step \ref{s5} of Algorithm \ref{s5} is implemented directly. 

The implementation for the case $d = 1$ is similar. We do not repeat all details here.

\subsection{Numerical examples}

We show two numerical results in 1D and two numerical results in 2D.

\subsubsection{Examples in 1D}

{\bf Test 1.}
We test if the convexification method can be applied to compute a periodic solution to a Hamilton-Jacobi equation.
We compute the solution to 
\begin{equation}
    6u(x) + \sqrt{|u'(x)|^2 + 1} = 
    6 e^{\sin(\pi x)}+\sqrt {
    \pi^2  
\cos^2 (\pi x)    e^{2\sin ( 
\pi x ) })+1}
\quad 
x \in \R.
\label{test1}
\end{equation}
The true solution to \eqref{test1} is the function $u^*(x) = e^{\sin(\pi x)}$, $x \in \R.$ 
The true and computed solution are given in Figure \ref{fig_1da}.

\begin{figure}[ht]
    \centering
    \subfloat[\label{fig_1a} The true and computed solution to \eqref{test1}]{\includegraphics[width = .3\textwidth]{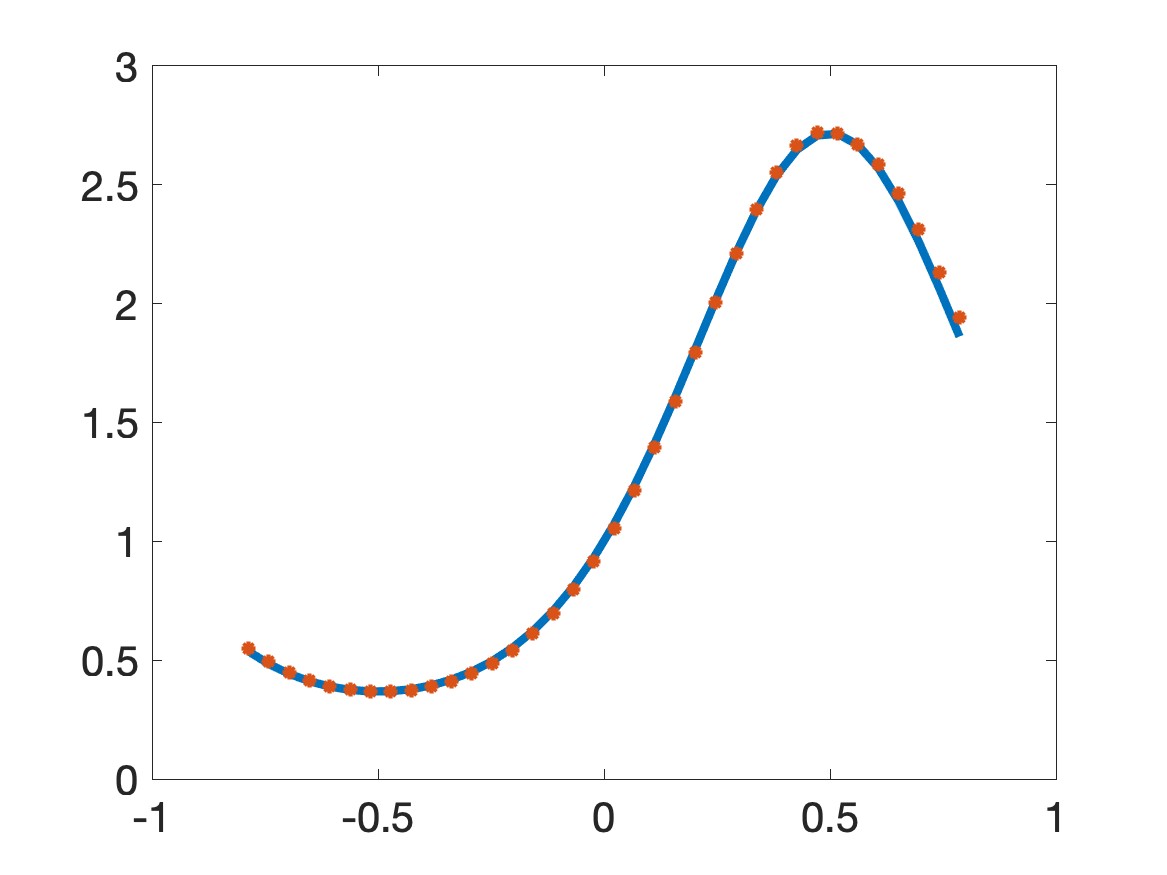}}
    \quad
    \subfloat[\label{fig_1b}  The relative error]{\includegraphics[width = .3\textwidth]{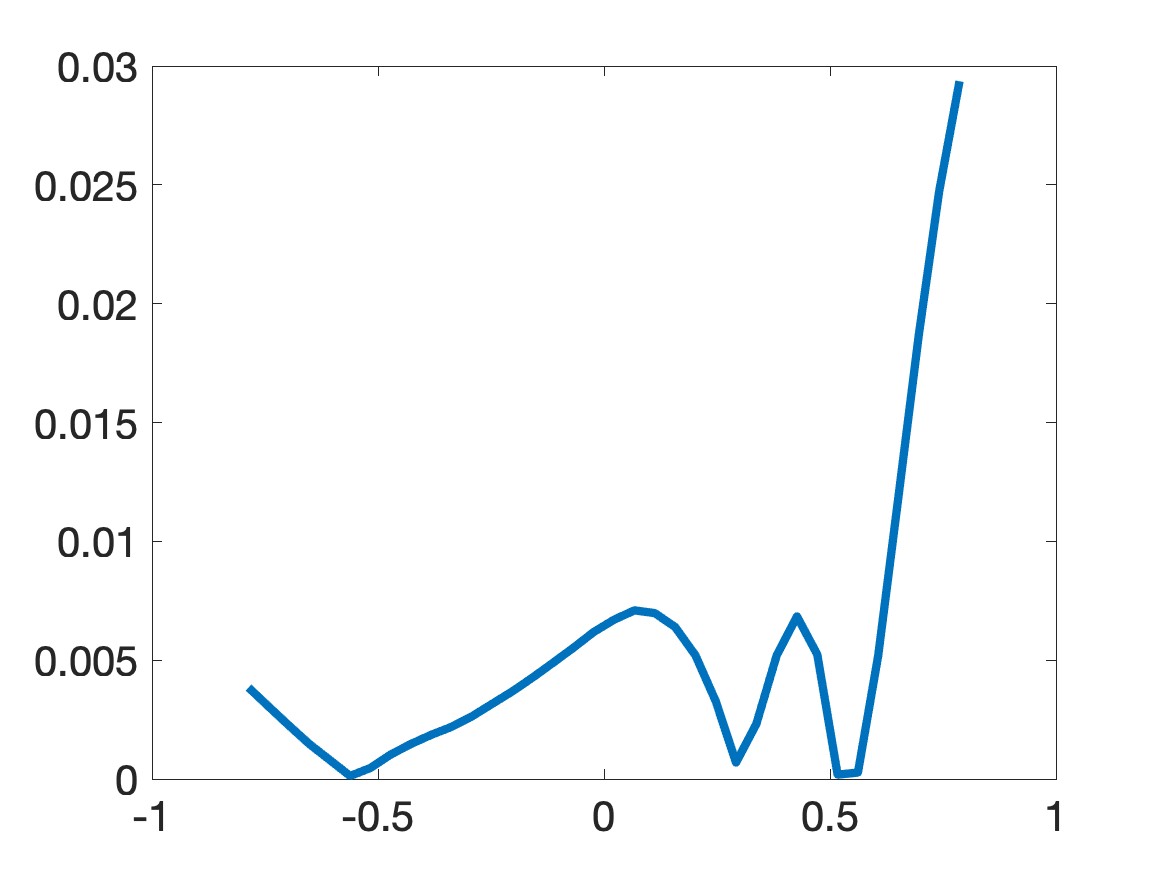}}
    \caption{Test 1. True (solid) and computed  (dot) solutions  to Hamilton-Jacobi equation \eqref{test1} in the interval $(-0.8, 0.8)$ and the relative error $\frac{|u_{\rm comp} - u^*|}{\|u^*\|_{L^{\infty}(G)}}$. The maximal value of this error function is $0.0294$.}
    \label{fig_1da}
\end{figure}

The convexification method provides a good solution to \eqref{test1}.The true solution in this test is periodic. Computing periodic solutions to Hamilton-Jacobi equations is very interesting and is a great concern in the scientific community; especially, in the study of periodic structure.
The numerical result is  satisfactory.
The error in computation is small.

{\bf Test 2} We next test the case when solution to \eqref{HJ} is quasi periodic. 
We solve the equation
\begin{equation}
    5u(x) + \sqrt{|u'(x)|^2 + 1} = 
    5 \sin \left( \frac{\pi x^4}{2} \right) +\sqrt {4\,{\pi}^{2}{x}^{6}
 \left( \cos \left( \frac{\pi x^4}{2} \right)  \right) ^{2}+1}    
\quad 
x \in \R.
\label{test2}
\end{equation}
The true solution to \eqref{test2} is $u^*(x) = \sin \left( \frac{\pi x^4}{2} \right)$ for all $\x \in \R.$
The graphs of the true solution $u^*$ and the computed solution bu using Algorithm \ref{alg} 
are displayed in Figure \ref{fig_2da}.

\begin{figure}[ht]
    \centering
    \subfloat[\label{fig_2a} The true and computed solution to \eqref{test2}]{\includegraphics[width = .3\textwidth]{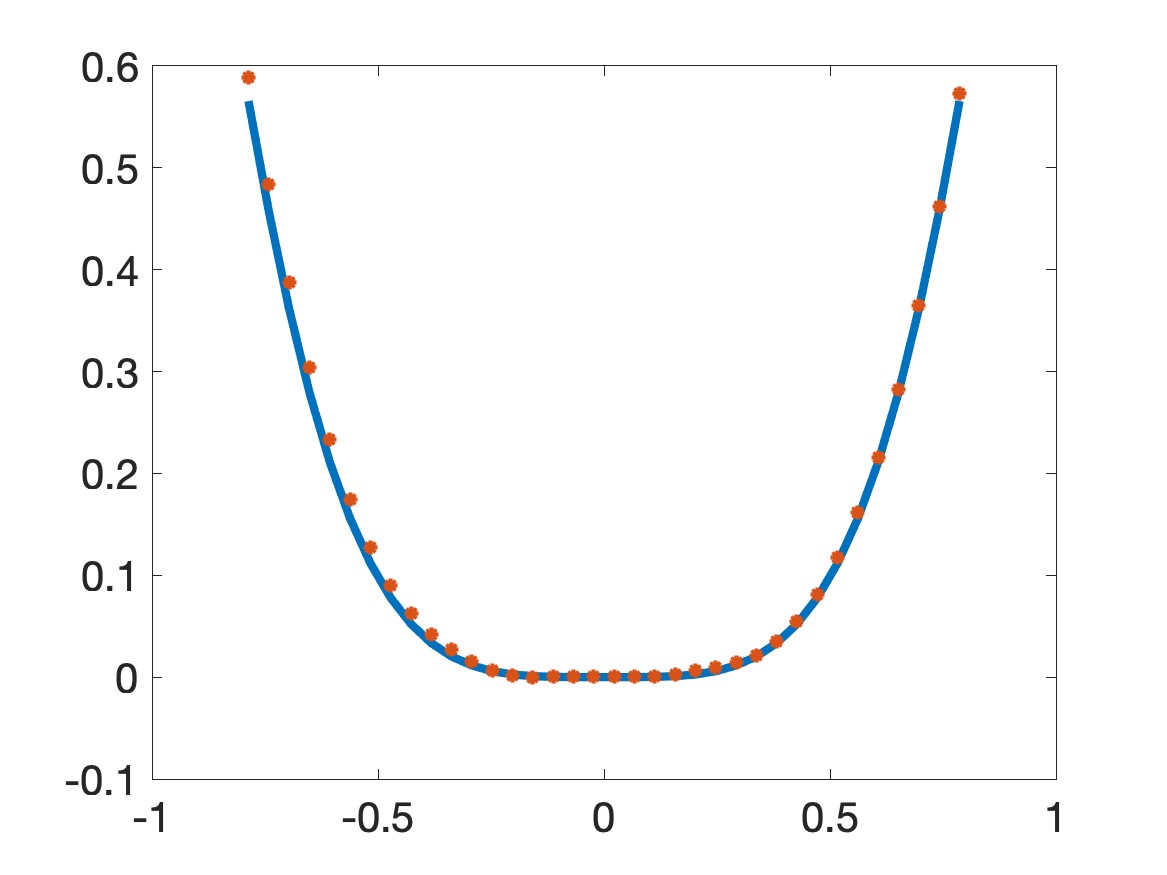}}
    \quad
    \subfloat[\label{fig_2b}  The relative error]{\includegraphics[width = .3\textwidth]{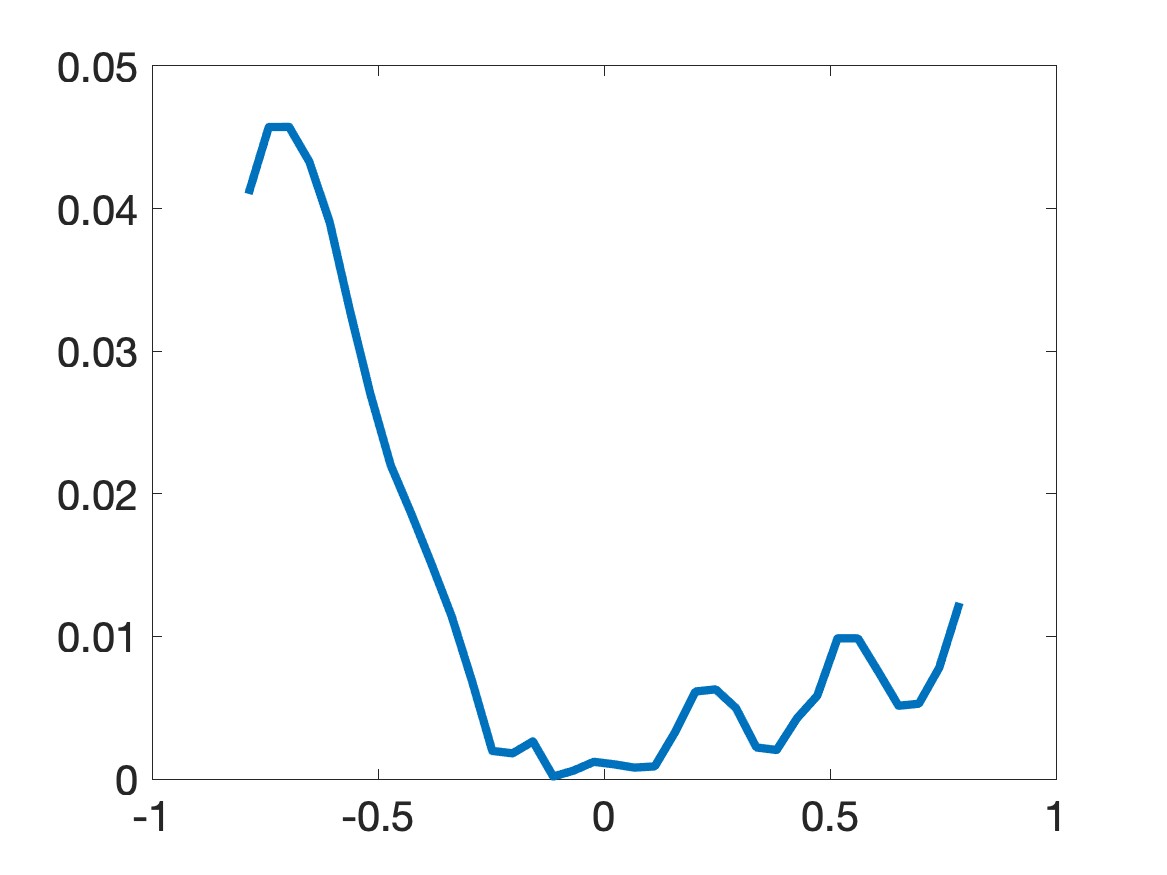}}
    \caption{Test 2. True (solid) and computed (dot) solutions to Hamilton-Jacobi equation \eqref{test2} in the interval $(-0.8, 0.8)$ and the relative error $\frac{|u_{\rm comp} - u^*|}{\|u^*\|_{L^{\infty}(G)}}$. The maximal value of this error function is $0.0457$.}
    \label{fig_2da}
\end{figure}

As in Test 1, it is evident that the convexification method delivers a satisfactory solution to \eqref{test2}. 
This test is interesting because the solution is  quasi-periodic. 
Computing this kind of solution that is not periodic is more interesting than the case of periodic solutions.

{\bf Test 3.} In Test 1 and Test 2, we study the case when the solution and the nonlinearity $H$ are smooth. We now test the nonsmooth case. 
We solve the equation
\begin{equation}
	10 u + \sqrt{|u'|^2 + 1} = g(x)
	\quad x \in \R
	\label{5.3333}
\end{equation}
where 
\[
	g(x) = \left\{
		\begin{array}{ll}
			10\big(-|2x| + \sin(x)\big)
			+\sqrt{(-2 + \sin(x))^2 + 1}
			&x \geq 0,\\
			10\big(-|2x| + \sin(x)\big)
			+\sqrt{(2 + \sin(x))^2 + 1} &x < 0.
		\end{array}
	\right.
\]
The true viscosity solution to \eqref{5.3333} is given by $u^*(x) = -|2x| + \sin(x)$ for all $x \in \R$.
In fact, we only need to verify the conditions in Definition \ref{def_viscos} at the corner of the graph of $u^*$, say at the place where $x_0 = 0$. Let $\varphi$ be a function in the class $C^1(\R)$ with $u^* - \varphi$ having a strict maximum at $x_0 = 0.$ Without lost of the generality, we can consider the case $u(0) = \varphi(0) = 0$. 
It is clear that $\varphi'(0) \in [-2, 2].$
So,
 $10\varphi(0) + \sqrt{|\varphi'(0)|^2 + 1} \leq \sqrt{5} = g(0).$
 Hence, $u^*$ is a subviscosity solution to \eqref{5.3333}. It is also a superviscosity supersolution to \eqref{5.3333} because there is no smooth function $\phi$ touches the function $u^*$ from below at $x_0 = 0.$

\begin{figure}[ht]
    \centering
    \subfloat[\label{fig_3a} The true and computed solution to \eqref{test1}]{\includegraphics[width = .3\textwidth]{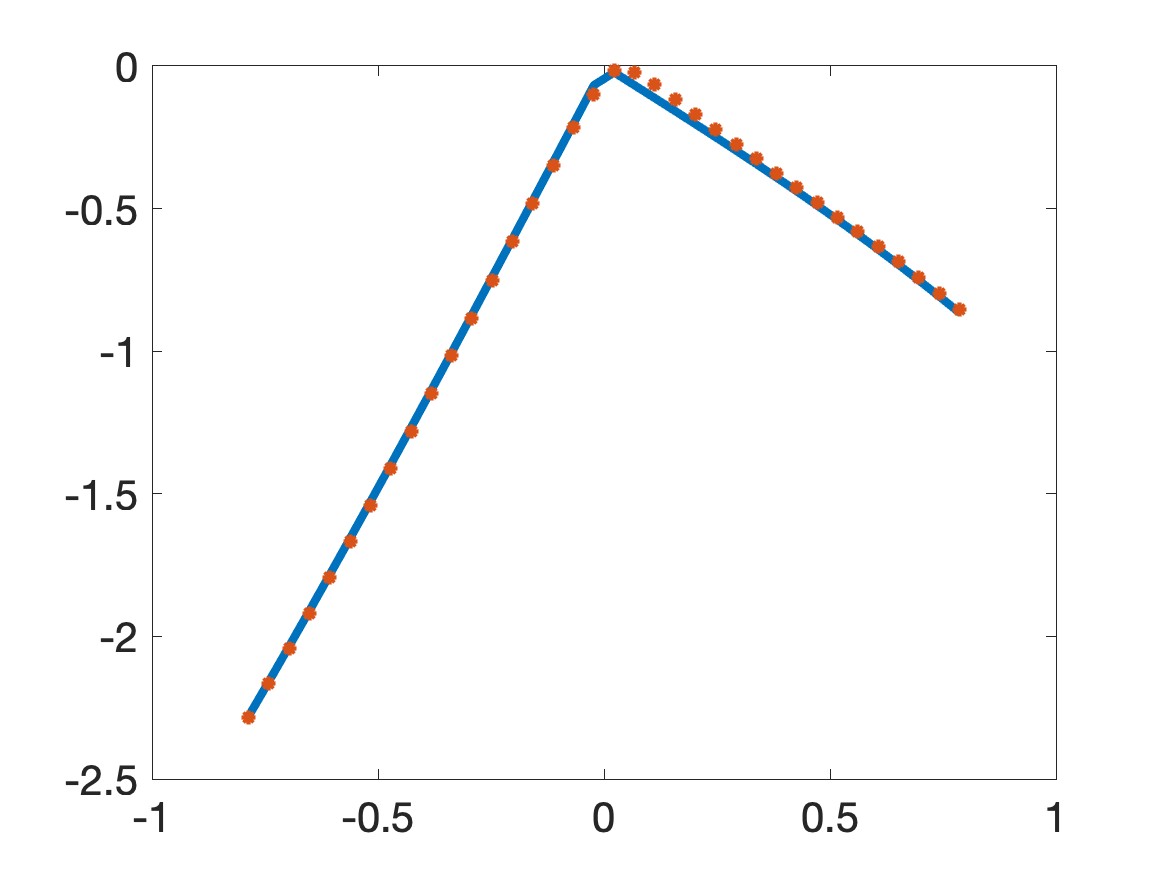}}
    \quad
    \subfloat[\label{fig_3b}  The relative error]{\includegraphics[width = .3\textwidth]{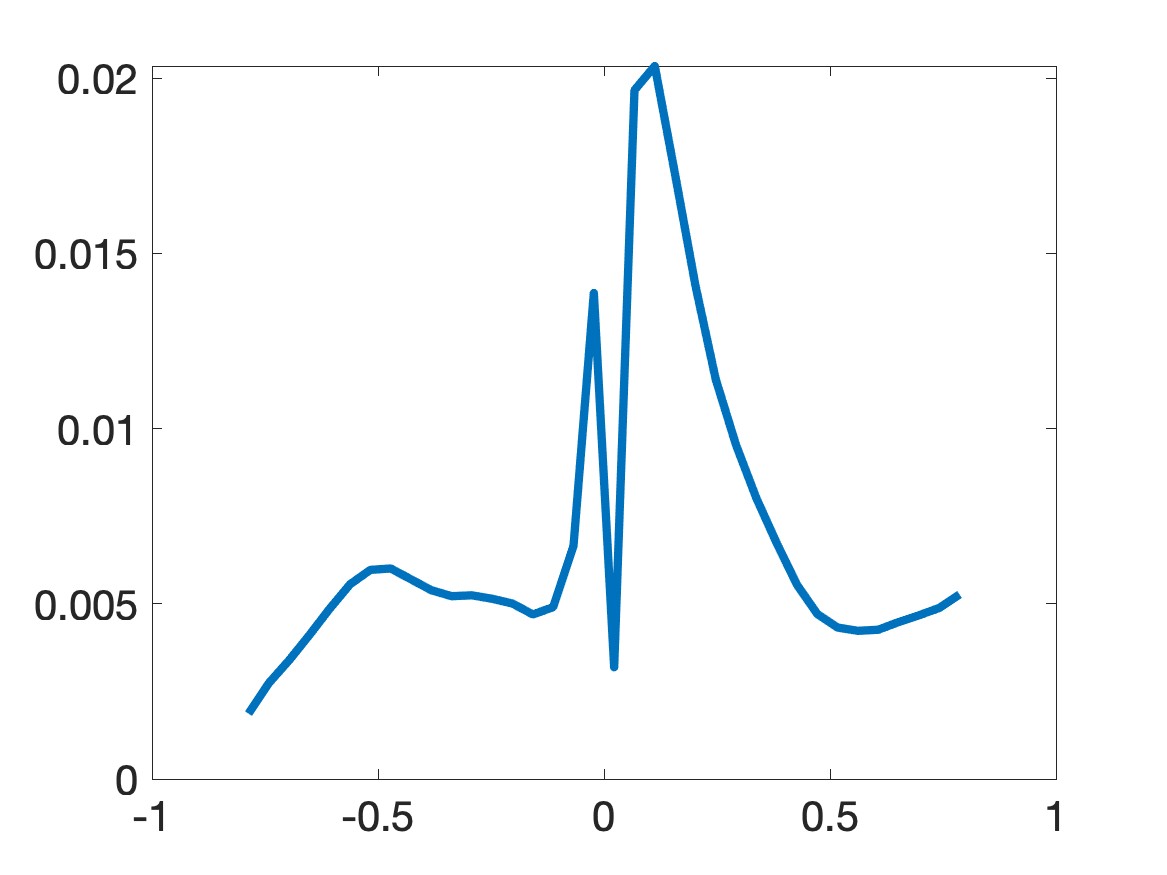}}
    \caption{\label{test3_1D}Test 3. True (solid) and computed  (dot) solutions  to Hamilton-Jacobi equation \eqref{test1} in the interval $(-0.8, 0.8)$ and the relative error $\frac{|u_{\rm comp} - u^*|}{\|u^*\|_{L^{\infty}(G)}}$. The maximal value of this error function is $0.0203$.}
    \label{fig_1da}
\end{figure}

Although this test is challenging, it is evident from Figure \ref{test3_1D} that the convexification method provides acceptable numerical result. The error occurs mostly at the discontinuity of the function $g$ and at the top corner of the graph of the solution.

\subsubsection{Examples in 2D}

{\bf Test 4.} We consider the case $d = 2$. We test the  convexification method by solving the following 2D Hamilton-Jacobi equation
\begin{multline}
    7 u(\x) + \sqrt{|\nabla u|^2 + 1} 
    = 
    7\,\sin \left( \frac{1}{2}\,\pi\, \left( {x}^{2}- \left( y- 0.2 \right) ^{2}
 \right)  \right) 
 +\frac{1}{2}\,\Big(4\,{\pi}^{2}{x}^{2} \left( \cos \left( 
\frac{1}{2}\,\pi\, \left( {x}^{2}- \left( y- 0.2 \right) ^{2} \right) 
 \right)  \right) ^{2}
 \\
 +{\pi}^{2} \left( -2\,y+ 0.4 \right) ^{2}
 \left( \cos \left( \frac{1}{2}\,\pi\, \left( {x}^{2}- \left( y- 0.2 \right) ^
{2} \right)  \right)  \right) ^{2}+4\Big)^{1/2}
 \label{test3}
\end{multline}
 for all $\x = (x, y) \in \R^2.$
The true solution to \eqref{test3} is $u^*(\x) = \sin\big(\frac{\pi}{2}(x^2 + (y - 0.2)^2)\big)$.
The numerical result of this test is displayed in Figure \ref{fig_3da}
\begin{figure}[ht]
    \centering
    \subfloat[ The true  solution to \eqref{test3}]{\includegraphics[width = .3\textwidth]{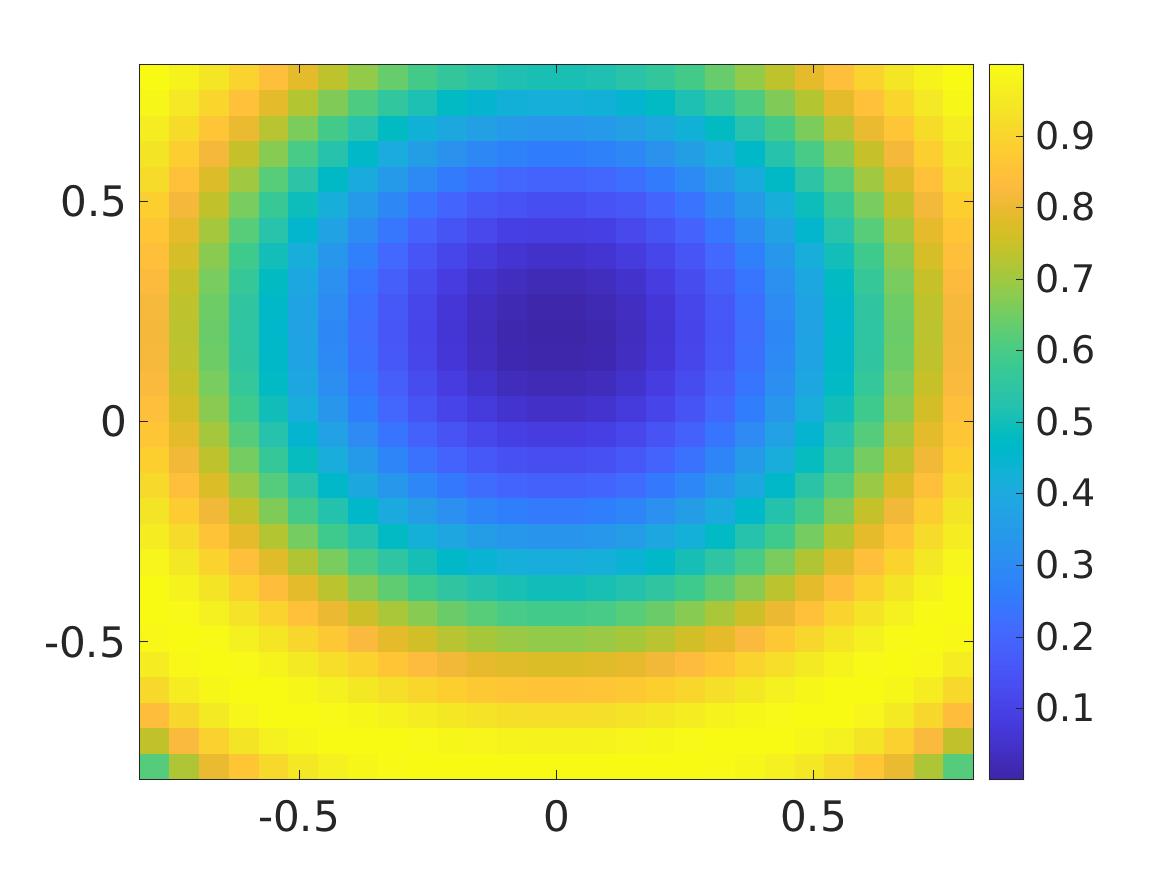}}
    \quad
    \subfloat[ The computed  solution to \eqref{test3}]{\includegraphics[width = .3\textwidth]{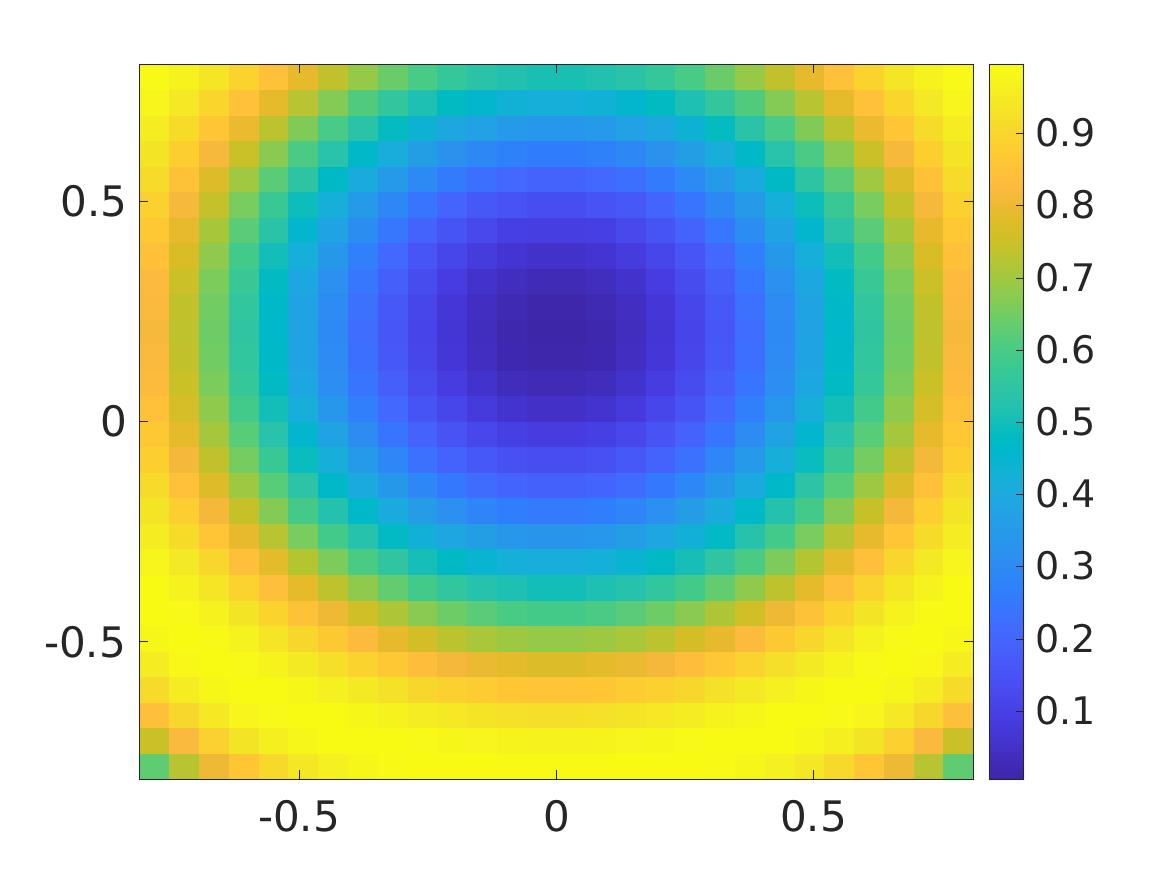}}
    \quad
    \subfloat[  The relative error]{\includegraphics[width = .3\textwidth]{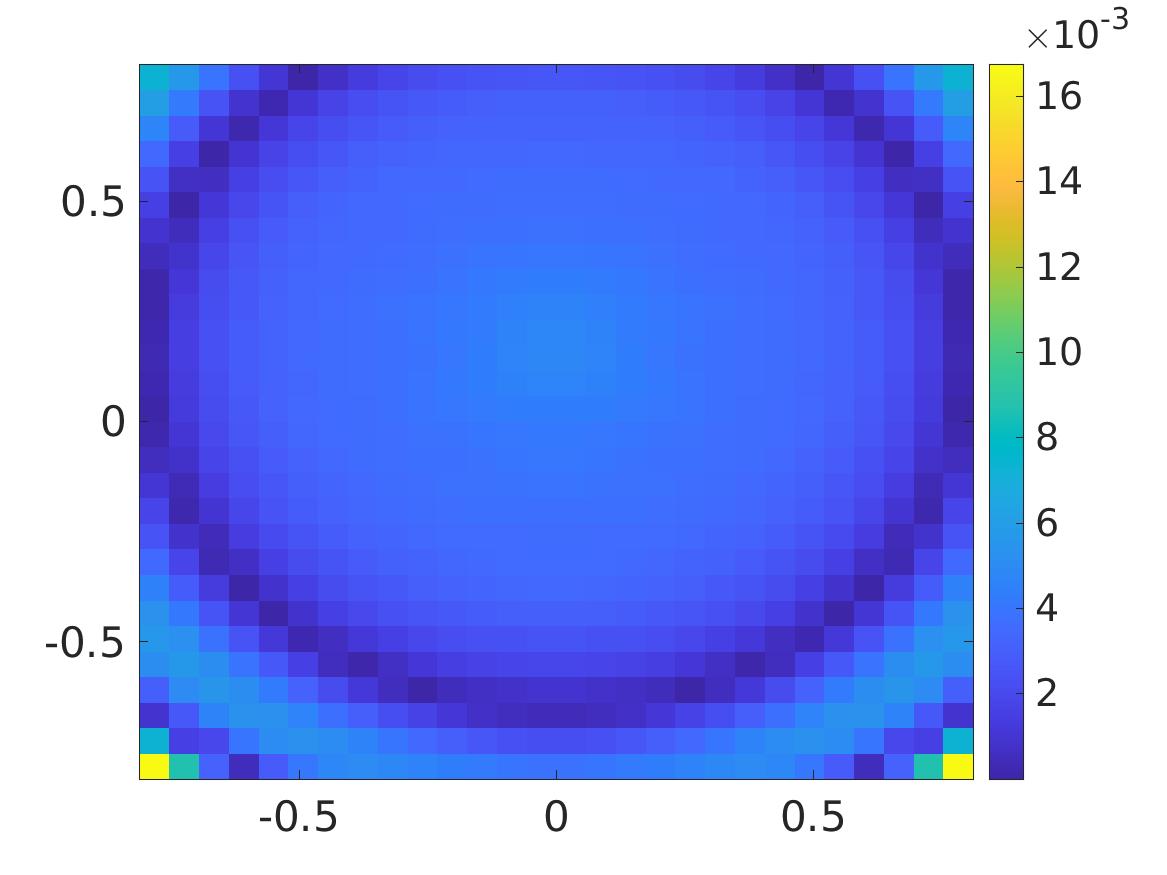}}
    \caption{Test 4. True and computed solutions to Hamilton-Jacobi equation \eqref{test3} in the interval $(-0.8, 0.8)^2$. The relative error in (c) is given by $\frac{|u_{\rm comp} - u^*|}{\|u^*\|_{L^{\infty}(G)}}$. The maximal value of this error function is $0.0168$.}
    \label{fig_3da}
\end{figure}

It is evident that the numerical result of this test is out of expectation. 
It is interesting to mention that the convexification method successfully compute the quasi-periodic solution to Hamilton-Jacobi equations.

{\bf Test 5.}  In this example, we test Algorithm \ref{alg} for unbounded and quasi-periodic solution.
More interestingly, in the test, the Hamiltonian is not convex with respect to $\nabla u.$
We solve the following 2D Hamilton-Jacobi equation
\begin{equation}
    10 u(\x) + |u_x(\x)| - |u_y(\x)| 
    = 
 -10\,x+10\,\cos \left( {x}^{2}+y \right) + \left| 1+2\,x\sin \left( {x
}^{2}+y \right)  \right| - \left| \sin \left( {x}^{2}+y \right) 
 \right|
 \label{test4}
\end{equation}
 for all $\x = (x, y) \in \R^2.$
The true solution to \eqref{test4} is $u^*(\x) = -x + \cos(x^2 + y)$.
The numerical result of this test is displayed in Figure \ref{fig_4da}
\begin{figure}[ht]
    \centering
    \subfloat[ The true  solution to \eqref{test4}]{\includegraphics[width = .3\textwidth]{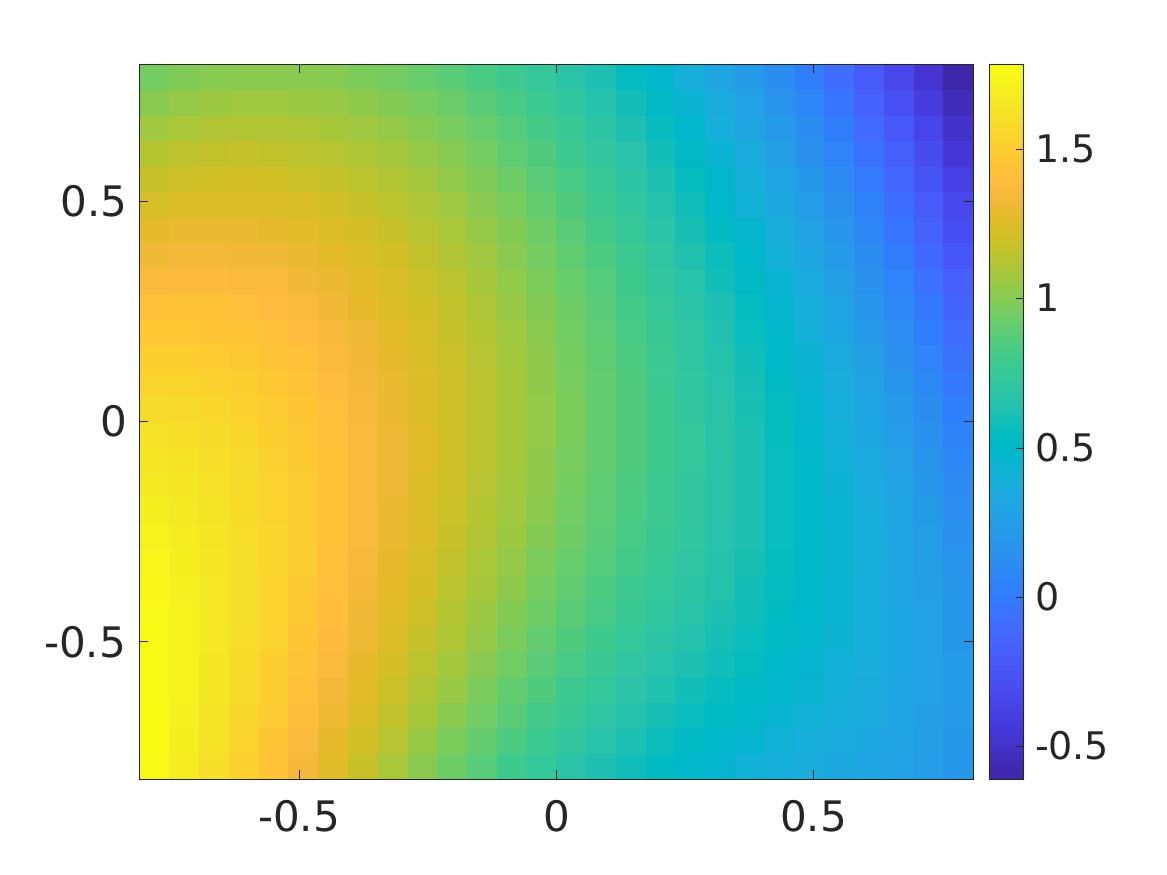}}
    \quad
    \subfloat[ The computed  solution to \eqref{test4}]{\includegraphics[width = .3\textwidth]{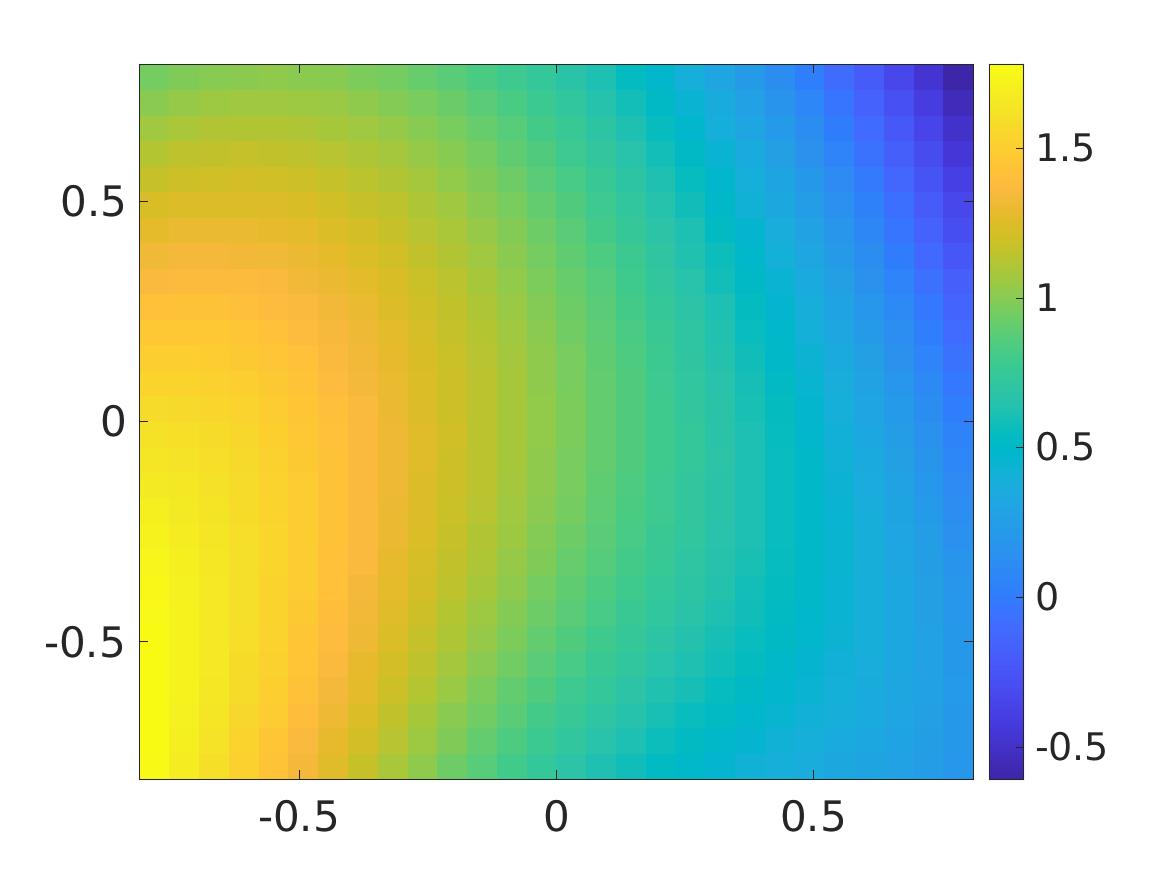}}
    \quad
    \subfloat[  The relative error]{\includegraphics[width = .3\textwidth]{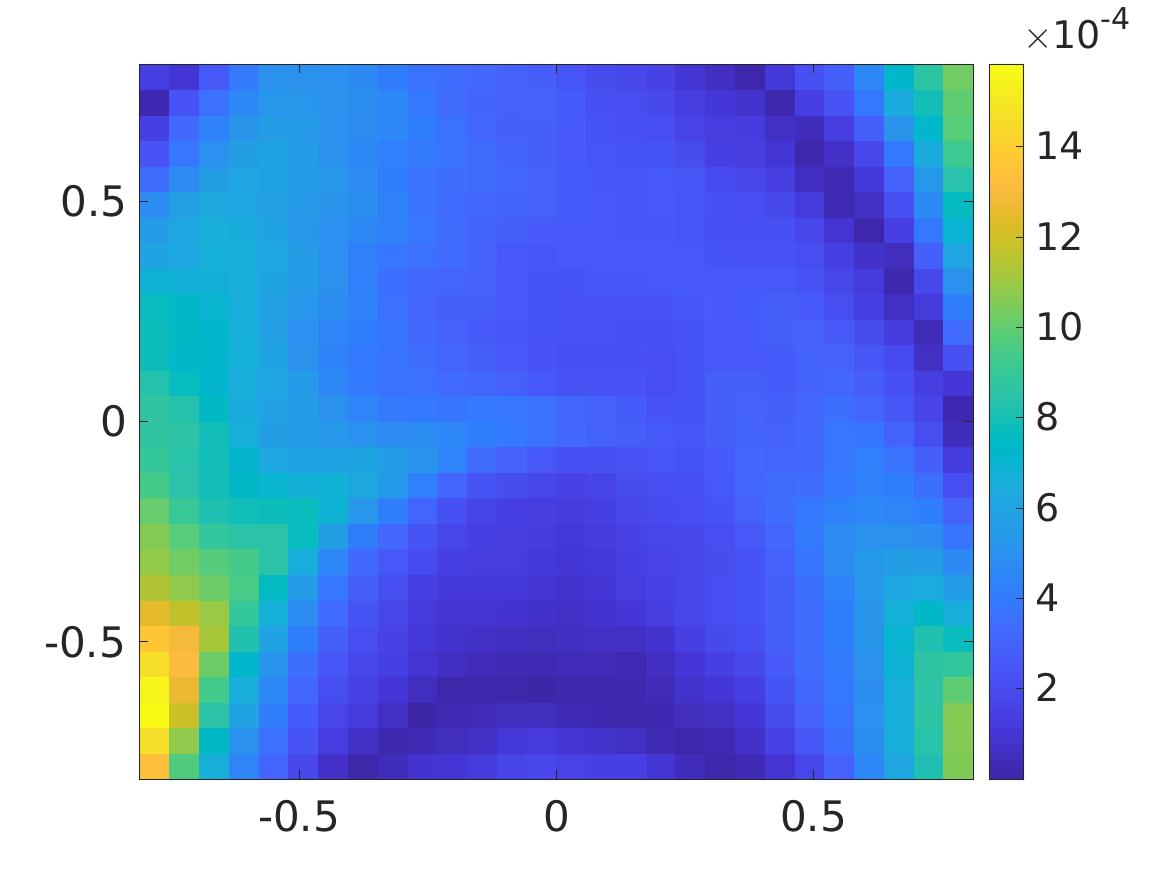}}
    \caption{Test 5. True and computed solutions to Hamilton-Jacobi equation \eqref{test4} in the interval $(-0.8, 0.8)^2$. The relative error in (c) is given by $\frac{|u_{\rm comp} - u^*|}{\|u^*\|_{L^{\infty}(G)}}$. The maximal value of this error function is $0.0016$.}
    \label{fig_4da}
\end{figure}

Although the solution to this test has an unbounded component and a quasi periodic component, we can compute the solution of this test with very small error.

{\bf Test 6.}  Like in Test 3, we consider a special Hamilton-Jacobi equation, in which the Hamiltonian is not convex. The true solution is not in the class $C^1$. 
We solve the equation
\begin{equation}
	10 u + |u_x| - |u_y| = g(\x)
	\label{test5}
\end{equation}
where
\[ 
g(\x) = \left\{
	\begin{array}{ll}
		10\big(-|2x| + \cos(x^2 + \pi y)\big)
		+ 2|  1 + 2x \sin(x^2 + \pi y)| 
		- \pi |\sin(x^2 + \pi y)| &x \geq 0, y \in \R
	\\	
	10\big(-|2x| + \cos(x^2 + \pi y)\big)
		+ 2|  1 - 2x \sin(x^2 + \pi y)| 
		- \pi |\sin(x^2 + \pi y)| &x < 0, y \in \R.
	\end{array}
\right.
\]
The true solution is given by
$u^*(\x) = -|2x| + \cos(x^2 + \pi y)$ 
for all $\x = (x, y) \in \R^2.$
In order to verify that $u^*$ is the viscosity solution to \eqref{test5}, we argue similarly to the argument in Test 3.
The numerical result of this test is displayed in Figure \ref{fig_5da}.
\begin{figure}[ht]
    \centering
    \subfloat[ The true  solution to \eqref{test5}]{\includegraphics[width = .3\textwidth]{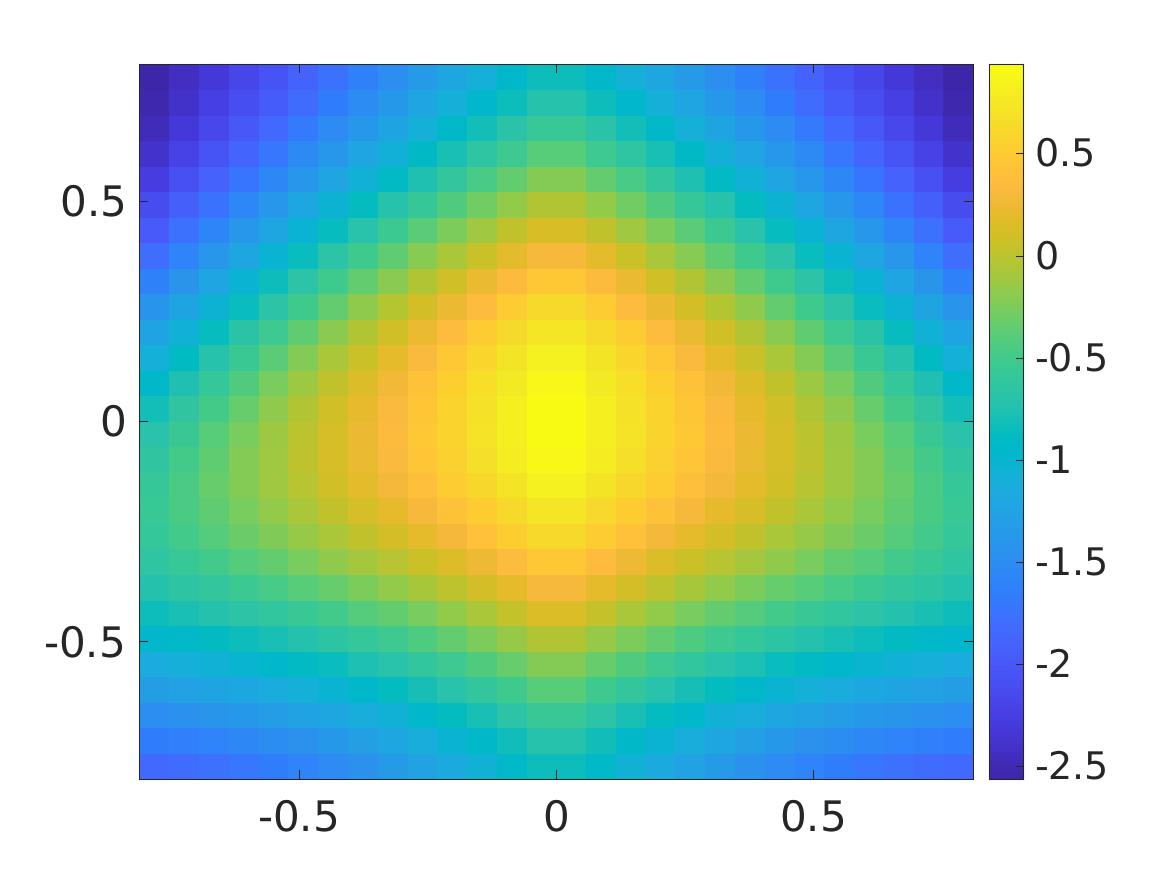}}
    \quad
    \subfloat[ The computed  solution to \eqref{test5}]{\includegraphics[width = .3\textwidth]{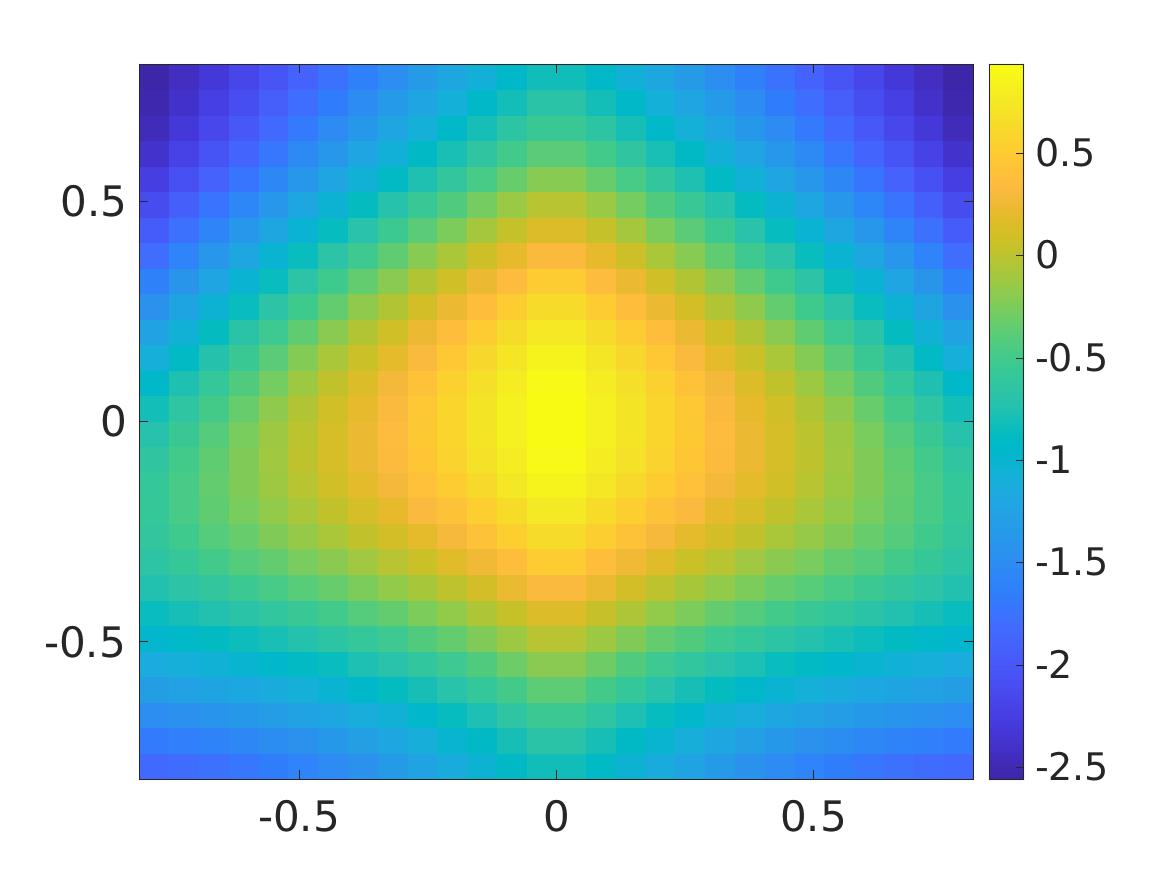}}
    \quad
    \subfloat[  The relative error]{\includegraphics[width = .3\textwidth]{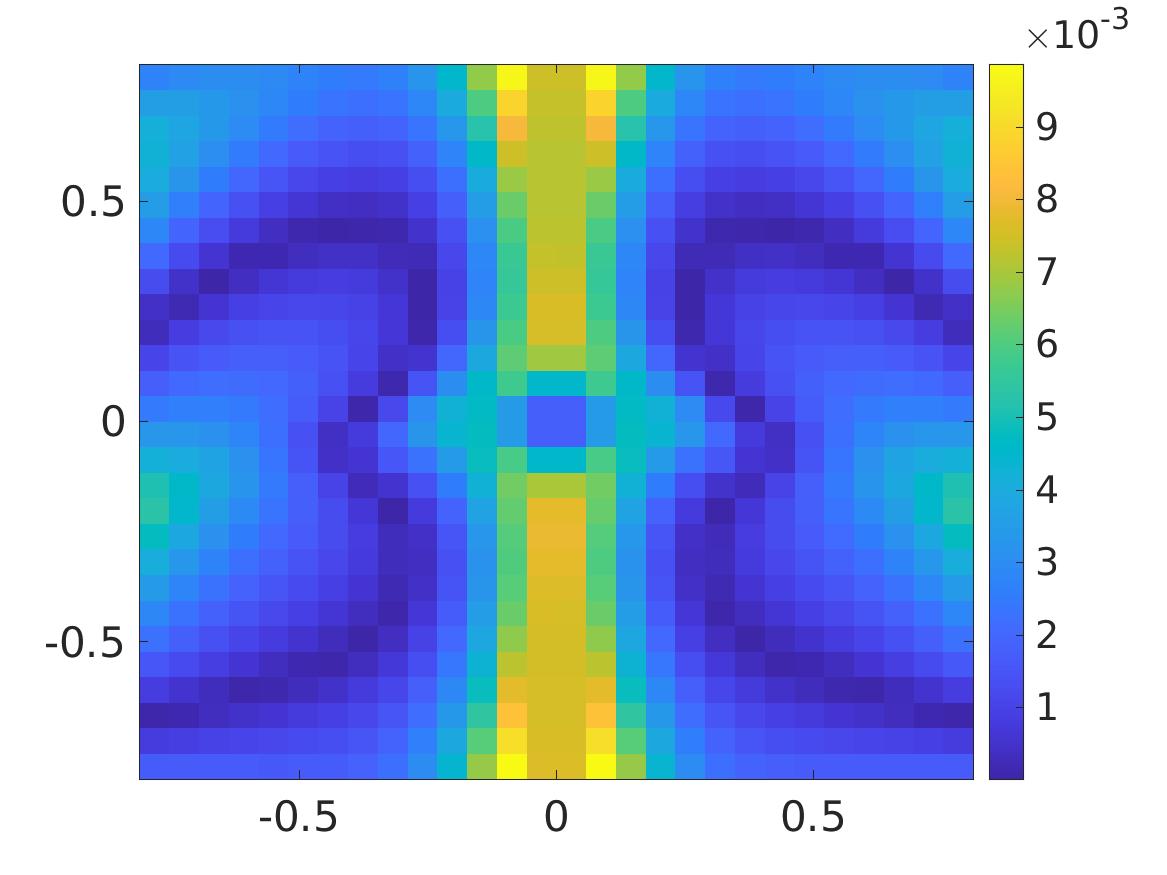}}
    \caption{Test 6. True and computed solutions to Hamilton-Jacobi equation \eqref{test5} in the interval $(-0.8, 0.8)^2$. The relative error in (c) is given by $\frac{|u_{\rm comp} - u^*|}{\|u^*\|_{L^{\infty}(G)}}$. The maximal value of this error function is $0.0099$.}
    \label{fig_5da}
\end{figure}

It is remarkable when observing that although the true solution is not in the class $C^1$, it can be computed. The error occurs in a neighborhood  the line $\{(x = 0, y)\}$ where $u$ is not differentiable. 


\section{Concluding remarks}\label{sec6}
In this paper, we have developed a new version of the Carleman based convexification method to compute the viscosity solutions to Hamilton-Jacobi equations on the whole space. 
Our procedure consists of two main stages. In Stage 1, we derive from the given Hamilton-Jacobi equation on $\R^d$ another Hamilton-Jacobi equation on a bounded domain by applying a truncation technique and a simple change of variable. 
It is important to mention that the boundary conditions for the Hamilton-Jacobi equation obtained Stage 1 cannot be exactly computed.
Only approximations ones are derived. 
This feature makes the original convexification method is not applicable.
In Stage 2, we develop the new version of the Carleman-based convexification method to solve the new Hamilton-Jacobi equation with approximated boundary conditions. 
The main theorems in this paper guarantee that the Carleman-based convexification method in Stage 2 delivers reliable numerical solutions to nonlinear Hamilton-Jacobi equations without requiring good initial guess.

\section*{Acknowledgement} This work  was partially supported  by National Science Foundation grant DMS-2208159, and
by funds provided by the Faculty Research Grant program at UNC Charlotte Fund No. 111272.


\end{document}